\pdfoutput=1

\documentclass[11pt,a4paper]{article}

\usepackage{amsmath}			
\usepackage{amsthm}
\usepackage{amssymb}
\usepackage{mathtools}
\usepackage[ruled,vlined]{algorithm2e} 
\usepackage{color} 
\usepackage[utf8]{inputenc}		
\usepackage{enumitem} 			
\usepackage{dsfont}
\usepackage[T1]{fontenc} 
\usepackage[super,negative]{nth}
\usepackage{hyperref}
\usepackage{cleveref}
\usepackage{graphicx}
\usepackage[outdir=./]{epstopdf}
\graphicspath{{../fig/}{../figbil/}{fig/}{figbil/}}
\usepackage{subcaption} 
\usepackage[round-precision=3,round-mode=figures]{siunitx}
\usepackage[toc,page]{appendix} 
\usepackage{accents}
\usepackage{import}
\usepackage{afterpage}
\usepackage{scrextend}
\usepackage[a4paper]{geometry}
\usepackage{xcolor}
\usepackage{chngcntr}
\usepackage{layout}	
\usepackage{placeins}

\afterpage{\FloatBarrier}

\DeclareCaptionLabelSeparator{protectedspace}{~}

\counterwithout{figure}{section}
\counterwithout{table}{section}

\definecolor{titlepagecolor}{cmyk}{1,.60,0,.40}
\definecolor{namecolor}{cmyk}{1,.50,0,.10} 

\definecolor{title}{cmyk}{1,.60,0,.40}

\definecolor{airforceblue}{RGB}{0, 73, 114}

\DeclareCaptionFont{title}{\color{title}}

\DeclareCaptionFont{blue}{\color{blue}}
\DeclareCaptionFont{black}{\color{black}}
\hypersetup{
    colorlinks,
    linkcolor={airforceblue},
    citecolor={green!50!black},
    urlcolor={blue!80!black}
}

\newtheoremstyle{dotlessP}{}{}{\color{black}}{}{\color{airforceblue}\bfseries}{}{ }{}
\theoremstyle{dotlessP}

\newtheorem{prop}{Proposition}[section]
\newtheorem*{prop*}{Proposition}
\newtheorem{thm}{Theorem}[section]
\newtheorem*{thm*}{Theorem}
\newtheorem{defn}{Definition}[section]
\newtheorem*{defn*}{Definition}
\newtheorem{lem}{Lemma}[section]
\newtheorem*{lem*}{Lemma}
\newtheorem{cor}{Corollary}[section]
\newtheorem*{cor*}{Corollary}

\newtheorem*{com*}{Comment}

\newtheorem*{prob*}{Comment}
\newtheorem{ex}{Example}[section]
\newtheorem*{ex*}{Example}

\newtheorem*{con*}{Condition}

\newtheorem*{optsys*}{Optimality system}

\theoremstyle{remark}
\newtheorem{rem}{Remark}[section]
\newtheorem*{rem*}{Remark}

\DeclareMathOperator{\dive}{\grad \cdot }
\DeclareMathOperator{\id}{I}

\DeclareMathOperator{\grad}{\nabla}

\DeclareMathOperator*{\argmin}{arg\,min}

\newcommand{\R}{\mathbb{R}}
\newcommand{\N}{\mathbb{N}}

\newcommand{\Ltwo}{L^2 (\Omega)}
\newcommand{\He}{{H^1 (\Omega)}}

\newcommand{\Hen}{H^{1}_0(\Omega)}

\newcommand{\Hk}{{H^k (\Omega)}}

\newcommand{\Hend}{H^{-1}(\Omega)}

\newcommand{\Linfty}{L^\infty (\Omega)}

\newcommand{\Uad}{U_{ad}}
\newcommand{\Fad}{F_{ad}}

\newcommand{\ydelta}{{y_\delta}}
\newcommand{\ydeltaj}{\ydelta_j}

\newcommand{\dual}{'}

\newcommand{\yex}{y^\dagger} 

\newcommand{\uex}{u^\dagger}

\newcommand{\un}{(u^n)}

\newcommand{\ubar}[1]{\underaccent{\bar}{#1}}

\newcommand{\adjoint}{^\ast}

\newcommand{\Jay}{\mathcal{J}_{\alpha,\ydelta}}
\newcommand{\Jayn}{\mathcal{J}_{\alpha^n,\ydelta^n}}

\newcommand{\Lpay}{\mathcal{P}_{\alpha,\ydelta}}
\newcommand{\Lpayon}{\mathcal{P}_{\alpha,\ydelta^n}}
\newcommand{\Lpayno}{\mathcal{P}_{\alpha^n,\ydelta}}
\newcommand{\Lpaynn}{\mathcal{P}_{\alpha^n,\ydelta^n}}
\newcommand{\Lpayast}{\mathcal{P}_{\alpha^\ast,\ydelta}}

\makeatletter
\newcommand{\LeftEqNo}{\let\veqno\@@leqno}
\makeatother

\makeatletter
\providecommand*{\input@path}{}
\g@addto@macro\input@path{{fig/}{figbil/}}
\makeatother

\newcommand{\ie}{i.e. }
\newcommand{\eg}{e.g. }
\newcommand{\sto}{s.t.}

\numberwithin{equation}{section}

\date{February 5, 2018}
\author{Karl Kunisch, Gernot Holler and Richard Barnard}
\title{A bilevel approach for learning the weights in multi-penalty Tikhonov regularization}

\begin{document}

\title{A Bilevel Approach for Parameter Learning in Inverse Problems}

\author{Gernot Holler\footnotemark[2]\ \footnotemark[5]
\and Karl Kunisch\footnotemark[4] \footnotemark[6]
\and Richard C. Barnard\footnotemark[3] \footnotemark[7]}
\pagestyle{myheadings}
\thispagestyle{plain}

\maketitle

\renewcommand{\thefootnote}{\fnsymbol{footnote}}

\footnotetext[2]{Institute for Mathematics and Scientific Computing, University of Graz,
  Heinrichstr. 36, 8010 Graz, Austria, email: \texttt{gernot.holler@uni-graz.at} .}

\footnotetext[4]{Institute for Mathematics and Scientific Computing, University of Graz, Heinrichstr. 36, 8010 Graz, Austria, and Radon
Institute, Austrian Academy of Sciences, Linz, Austria, email:
\texttt{karl.kunisch@uni-graz.at}
.}
\footnotetext[3]{Oak Ridge National Laboratory, Oak Ridge, TN 37831, USA, email: \texttt{barnardrc@ornl.gov} .}

\footnotetext[5]{The author gratefully acknowledges support from
the International Research Training Group IGDK1754, funded by the DFG and FWF.}

\footnotetext[6]{The author acknowledges partial support by the  ERC advanced grant 668998 (OCLOC) under the EU's H2020
research program.}

\footnotetext[7]{
This manuscript has been authored by UT-Battelle, LLC, under contract DE-AC05-00OR22725 with the US Department of Energy (DOE). The US government retains and the publisher, by accepting the article for publication, acknowledges that the US government retains a nonexclusive, paid-up, irrevocable, worldwide license to publish or reproduce the published form of this manuscript, or allow others to do so, for US government purposes. DOE will provide public access to these results of federally sponsored research in accordance with the DOE Public Access Plan (\texttt{http://energy.gov/downloads/doe-public-access-plan}).}

\renewcommand{\thefootnote}{\arabic{footnote}}

\maketitle

\newcommand{\keywords}[1]{\textbf{Key words.}\quad #1}

\begingroup\emergencystretch=.3em
\begin{abstract}
A learning approach to selecting regularization parameters in multi-penalty Tikhonov regularization is investigated. It leads  to a bilevel optimization problem, where the lower level problem is a Tikhonov regularized problem parameterized in the regularization parameters. Conditions which ensure the existence of solutions to the bilevel optimization problem of interest are derived, and these conditions are verified for two relevant examples. Difficulties arising from the possible   lack of convexity of the lower level problems are discussed. Optimality conditions are given provided that a reasonable constraint qualification holds. Finally, results from numerical experiments used to test the developed theory are presented.

\end{abstract}
\endgroup
\keywords{parameter learning, Tikhonov regularization, bilevel optimization, multi-penalty regularization}

\section{Introduction}

Tikhonov regularization is a well-known method for solving ill-posed inverse problems, see \eg \cite{tikhonov1977solutions,engl1996regularization,louis2013inverse,kaltenbacher2008iterative}.
Given only a noisy measurement $\ydelta$ of some outcome $\yex \in Y$, and assuming that the inverse problem is to find $\uex \in \Uad$ such that
\begin{equation}
S(\uex) = \yex,
\end{equation}
where $S$ is a mapping from a subset $\Uad$ of a Banach space $U$ to a Banach space $Y$, the Tikhonov regularized problem consists in solving
\begin{equation}\tag{$\mathcal{P}_{\alpha,\ydelta}$}\label{prob:llptest}
\min_{u \in \Uad} \mathcal{J}_{\alpha,\ydelta}(u) \equiv \, \| S(u) - \ydelta \|^2 +  \alpha \cdot \Psi(u),
\end{equation}
for suitable choices of a norm $\| \cdot \|$, a vector valued penalty function $\Psi \colon U \to [0,\infty]^r$, and a vector of regularization parameters $\alpha \in (0,\infty)^r$. Which norm and penalty functions should be chosen depends heavily on the specific application. For choosing the regularization parameters,  many general strategies have been proposed; see \eg \cite{engl1996regularization}, and the references given there. Typically these  strategies  focus on  the case of a single scalar regularization parameter, and they become quite  involved when one has to deal with a larger number of parameters.

\paragraph{The learning problem}
In this paper we consider a basic learning approach for selecting regularization parameters in \eqref{prob:llptest}. The idea is to choose regularization parameters based on their performance on a training database. In the simplest case, the database consists of a single vector of data $(\yex, \uex, \ydelta)$, where $\ydelta$ is a noisy measurement of $\yex$, and $\uex$ is such that
\begin{equation*}
S(\uex) = \yex.
\end{equation*}
We may think of $(\yex,\uex)$ as an idealistic ground truth input-output pair, and of $\ydelta$ as the associated  noisy measurement of the output available in practice. Given such data, for every choice of $\alpha$ we can compute the distance between solutions $u_\alpha$ to the regularized problem \eqref{prob:llptest} and the exact solution $\uex$. This is used in the learning process where we aim at finding the regularization parameter $\alpha^\ast$ for which a solution $u_{\alpha^\ast}$ to $(\mathcal{P}_{\alpha^\ast,\ydelta})$ has the minimal distance to $\uex$ over all parameter vectors within an a-priori chosen parameter set $I$. This leads us to the following problem:
\begin{equation}
\underset{\alpha \in I} {"\min"} \
\|\uex-u_\alpha\|^2 \quad  \text{s.t.} \quad u_\alpha \in \argmin_{u \in \Uad}\| S(u) - \ydelta \|^2 +  \alpha \cdot \Psi(u).
\end{equation}
The quotation marks are used, since if solutions to the Tikhonov regularized problems \eqref{prob:llptest} are not unique, then it is not clear which solutions to choose. One possibility is to look for $\alpha$ such that the minimal distance to the exact solution over all solutions to \eqref{prob:llptest} is small. This is called the optimistic position and leads to the following problem.
\begin{equation}\label{prob:optimistic}
\underset{\alpha \in I} {\min} \ \underset{u_\alpha \in \Uad} {\min}
\|\uex-u_\alpha\|^2 \quad  \text{\sto} \quad u_\alpha \in \argmin_{u \in \Uad} \| S(u) - \ydelta \|^2 +  \alpha \cdot \Psi(u).
\end{equation}
Another possibility is to look for $\alpha$ such that the maximal distance to the exact solution over all solutions to \eqref{prob:llptest} is small. This is called the pessimistic position and amounts to the following problem.
\begin{equation}
\underset{\alpha \in I} {\min} \ \underset{u_\alpha \in \Uad} {\max}
\|\uex-u_\alpha\|^2 \quad  \text{\sto} \quad u_\alpha \in \argmin_{u \in \Uad} \| S(u) - \ydelta \|^2 +  \alpha \cdot \Psi(u).
\end{equation}
Here we only consider the optimistic position \eqref{prob:optimistic}.
From now on we call \eqref{prob:optimistic} the learning problem, since by solving it regularization parameters should be learned. Conceptually, the learning problem is an optimization problem in two variables, which is  constrained by requiring that one variable is a solution of another optimization problem depending on the other variable. In the literature problems of this type are called bilevel optimization problems; see \eg \cite{dempe2002foundations}.

The present work was motivated by a similar learning approach that  has  successfully been used for imaging problems in \cite{Kunisch_Bilevel2013} and the subsequent works \cite{de2013image,ochs2015bilevel,DELOSREYES2016464}. In these works, both the cases of  smooth and  non smooth lower level problems are studied in a finite and infinite dimensional setting. However, in all these contributions it is required that $S$ is either an identity embedding operator or has closed range. The case of a general linear operator $S$ is considered in \cite{chung2017} in a finite dimensional setting. We are very much aware of the potential which may rest in currently heavily investigated technology of deep learning  in order to choose regularization parameters for inverse problems, and we aim to work in this direction. We hope that a mathematical analysis of the deep learning approach  can profit from the present work.

What we aim to do is to use parameter learning for the inverse problems of determining coefficients or controls in partial differential equations. This requires us to consider an infinite dimensional setting with $S$ either  a linear operator with non closed range or even  non linear.
Although we develop the theory in a somewhat general setting, throughout this work we have two concrete examples in mind. In the first example, $S$ is the linear solution operator to
\begin{equation}
-\gamma \Delta y+y = u \quad \text{in} \ \Omega, \quad \text{and} \quad y= 0 \quad \text{on} \ \partial \Omega,
\end{equation}
where $\gamma>0$. In the second example, $S$ is the non-linear solution operator to
\begin{equation}
-\dive \, ( u \grad y ) = f \quad \text{in} \ \Omega, \quad \text{and} \quad y= 0  \quad \text{on} \ \partial \Omega,
\end{equation}
where $f \in \Ltwo$ is given. In both examples $\Omega$ is assumed to be a bounded Lipschitz domain in $\R^d$, where $d \in \N$.

Let us now give a brief summary of the contents of the following sections. In \Cref{section:probstatement} we provide a precise statement of the learning problem and introduce the basic notation. In \Cref{sec:thelowerlevelproblem} we recall some basic properties of the lower level problem, \ie of the Tikhonov regularized problem. These properties are in turn used in \Cref{sec:exsol} to show that the learning problem has a solution under standard assumptions. In \Cref{sec:optcond} we discuss the derivation of optimality conditions for the learning problem. Standard examples for possible applications are presented in \Cref{sec:examples}. Finally, in \Cref{sec:numexp} we present results from numerical experiments.

\section{Problem statement}\label{section:probstatement}

In the following we present the general setting of the learning problem to be considered in this work.
\begin{equation}\LeftEqNo\label{prob:genbiopt}\tag{$\mathcal{LP}$}
\begin{cases}
\underset{\alpha \in [\ubar{\alpha},\bar{\alpha}],\, (y_\alpha,u_\alpha) \in Y \times \Uad} {\min}
\|  u_\alpha - \uex \|_{\tilde{U}}^2 \qquad  \text{subject to}
        \\ \\
        (y_\alpha,u_\alpha) \in \argmin\limits_{\substack{u \in \Uad \\ y \in Y}} \bigg\{\frac{1}{2m} \sum\limits_{j=1}^m  \|  y-\ydeltaj\|_{\tilde{Y}}^2 + \alpha \cdot \Psi(u) \mid e(y,u) = 0 \bigg\},
        \end{cases}
\end{equation}
where $m, r \in \N$, and
\begin{itemize}
\item $\Uad$ is a subset of a reflexive Banach space $U$,
\item $Y$ is a reflexive Banach space,
\item $\tilde{U}$ is a Hilbert space such that $U$ is continuously embedded in $\tilde{U}$,
\item $\tilde{Y}$ is a Hilbert space such that $Y$ is continuously embedded in $\tilde{Y}$,

\item $\uex \in \tilde{U}$ is the exact control, and $\ydeltaj \in \tilde{Y}$, $1 \leq j \leq m$, are noisy measurements of the exact state,
\item $e \colon Y \times \Uad \to Z$ represents equality constraints in a Banach space Z,
\item $\Psi_i \colon U \to [0,\infty]$, $1 \leq i \leq r$, are penalty functionals, and
\begin{equation*}
\Psi \coloneqq (\Psi_1,\dots,\Psi_r)^T,
\end{equation*}

\item  $\ubar{\alpha},\bar{\alpha} \in \R^r$ are bounds for the regularization parameters with
\begin{equation*}
0 < \ubar{\alpha} \leq \bar{\alpha} < \infty,
\end{equation*}
where the inequalities should be understood element wise.

\end{itemize}
Instead of working with an explicit solution operator $S$ as in the introduction, here we consider a more general implicit formulation by requiring that for feasible $(y,u) \in Y \times \Uad$ it holds that
\begin{equation}\label{eq:implicit}
e(y,u)= 0.
\end{equation}
If for each $u \in \Uad$ there exists a unique $y \in Y$ such that \eqref{eq:implicit} holds, then a solution operator $S$ can be defined by setting
\begin{equation*}
y=S(u) \quad \text{if and only if} \quad e(y,u)=0 \quad \text{for } (y,u) \in Y \times \Uad.
\end{equation*}
The so-called lower level problem
 \begin{equation}\LeftEqNo\tag{$\Lpay$}\label{prob:llpproblem}
\begin{cases}
\underset{(y,u) \in Y \times U}{\min} \Jay(y,u)  \equiv \frac{1}{2m} \sum\limits_{j=1}^m \| y - \ydeltaj\|_{\tilde{Y}}^2 +  \alpha \cdot \Psi(u) \quad \text{subject to} \\
u \in \Uad \quad \text{and} \quad e(y,u)=0, \\
\end{cases}
\end{equation}
which depends on the parameter $\alpha \in [\ubar{\alpha}, \bar{\alpha}]$, is a multi-penalty Tikhonov regularized inverse problem. We let
\begin{equation*}
\Fad \coloneqq \{ (y,u) \in Y \times U \mid u \in \Uad \text{ and } e(y,u)=0 \}
\end{equation*}
denote the set of feasible points of the lower level problem. To fix ideas, typical choices for the used spaces are
\begin{equation*}
U = \He, \quad Y = \Hen, \quad \tilde{Y}=\Ltwo, \quad \tilde{U}=\Ltwo,
\end{equation*}
where $\Omega$ is a bounded Lipschitz domain. Concrete examples are given in \Cref{sec:examples}.

\subsection{Basic assumptions}
The following assumptions are frequently invoked throughout this work.

\begin{enumerate}[label=$(\text{H}\arabic*)$]
\item The feasible control set $\Uad$ is convex and closed in $U$. \label{cond:Uadccc}
\item The feasible set of the lower level problem $\Fad$ is non-empty. \label{cond:nonempty}
\item For every sequence $(y^n, u^n)$ in  $Y \times \Uad$ and $(\bar{y}, \bar{u})  \in Y \times \Uad$ such that
\begin{equation*}
e(y^n, u^n) = 0 \quad \text{for all } n \in \N, \quad \text{and} \quad (y^n, u^n) \rightharpoonup (\bar{y}, \bar{u}), 
\end{equation*}
it follows that
\begin{equation*}
e(\bar{y}, \bar{u}) = 0.
\end{equation*}\label{cond:weakweak}

\item For every sequence $(y^n, u^n)$ in $\Fad$ it holds that if $\un$ is bounded in $U$, then $(y^n)$ is bounded in $Y$. \label{cond:boundsolop}

\item The function
\begin{equation*}
\sum_{i=1}^r \Psi_i
\end{equation*}
is coervice on $U$ and proper on $\Fad$. \label{cond:coer}

\item The penalty functionals $\Psi_i$, $1\leq i \leq r$, are weakly lower semi continuous on $U$. \label{cond:weakly}

\end{enumerate}
\section{The lower level problem}\label{sec:thelowerlevelproblem}

When we discuss existence of solutions and optimality conditions for the learning problem in \Cref{sec:exsol} and \ref{sec:optcond}, respectively, we frequently make use of basic properties of the lower level problem. In this section these properties are derived. Throughout this section we always assume that $\alpha \in [\ubar{\alpha}, \bar{\alpha}]$.

\subsection{Existence of solutions}
\begin{prop}[Existence of solutions]\label{thm:llp}
If \ref{cond:Uadccc}--\ref{cond:weakly} hold, then \eqref{prob:llpproblem} has a solution.
\end{prop}
\begin{proof}
By \ref{cond:Uadccc} the set $Y \times \Uad$ is closed and convex, and thus weakly closed \cite[Theorem 3.7 on p.60]{brezis2010functional}. It is then a direct consequence of \ref{cond:weakweak} that $\Fad$ is weakly sequentially closed. From \ref{cond:boundsolop}--\ref{cond:coer} and the assumption that $\alpha>0$, it follows that $\Jay$ is coercive on $\Fad$. The mapping
\begin{equation*}
(y,u) \mapsto \frac{1}{2m} \sum\limits_{j=1}^m \| y - \ydeltaj\|_{\tilde{Y}}^2
\end{equation*}
is weakly lower semi continuous as a convex continuous function \cite[Corollary 3.9 on p.61]{brezis2010functional}. In combination with \ref{cond:weakly} this implies that $\Jay$ is weakly lower semi continuous on $Y \times U$. Since it is well-known that a weakly lower semi continuous and coercive function attains a minimum on a non empty and weakly sequentially closed subset of a reflexive Banach space, the proof is complete.
\end{proof}
\begin{rem}
As an immediate consequence of \Cref{thm:llp}, we obtain that the feasible set of the learning problem \eqref{prob:genbiopt} is non empty.
\end{rem}

\subsection{Stability}

One of the reasons for regularizing an inverse problem is lack of stability with respect to the data. It is thus expected that stability, at least in some sense, holds for the Tikhonov regularized problem \eqref{prob:llpproblem}. Indeed, as stated below in \Cref{cor:stabilitytik}, stability can be guaranteed under reasonable assumptions. Before we begin working towards this result, we need to clarify what we mean by stability (in particular in the context of problems with possibly non unique solutions).

\begin{defn}[Stability with respect to the data]\label{defn:stabilitydata}
We say that $(\Lpay)$ is stable with respect to the data if and only if the following holds: For every sequence $(\ydelta^n)$ in $\tilde{Y}^m$ such that
\begin{equation*}
\ydelta^n \to \ydelta,
\end{equation*}
it follows that every sequence $(y^n,u^n)$ of corresponding solutions to $(\Lpayon$) has a cluster point, and every such cluster point is a solution to \eqref{prob:llpproblem}.

\end{defn}

\begin{rem}
If \eqref{prob:llpproblem} has a unique solution, then it is straightforward to verify that stability with respect to the data is equivalent to requiring that every sequence $(y^n,u^n)$ as in \Cref{defn:stabilitydata} is converging to the unique solution of \eqref{prob:llpproblem}.
\end{rem}

Recall that in the learning problem we minimize the distance to the exact control over the set of all feasible regularization parameters and corresponding solutions to the lower level problem. It is useful to know, if the lower level problem is stable with respect to the regularization parameters.

\begin{defn}[Stability with respect to the regularization parameters]\label{defn:stabilityparameters}
We say that \eqref{prob:llpproblem} is stable with respect to the regularization parameters if and only if the following holds: For every sequence $(\alpha^n)$ in $[\ubar{\alpha}, \bar{\alpha}]$ such that
\begin{equation*}
\alpha^n \to \alpha,
\end{equation*}
it follows, that every sequence $(y^n,u^n)$ of corresponding solutions to $(\Lpayno)$ has a cluster point, and every such cluster point is a solution to $(\Lpay)$.
\end{defn}
As a first step towards showing stability, we prove the following lemma, which states that under standard assumptions at least weak stability can be guaranteed with respect to both the data and the regularization parameters.

\begin{lem}[Weak stability]\label{lem:stabilityLLP}
Assume that \ref{cond:Uadccc}--\ref{cond:weakly} hold, and let $(\alpha^n,y_{\delta}^n)$ be a sequence in $ [\ubar{\alpha}, \bar{\alpha}] \times \tilde{Y}^m$ such that
\begin{equation*}
(\alpha^n,y_{\delta}^n) \to (\alpha,\ydelta).
\end{equation*}
Then every sequence $(y^n, u^n)$ of solutions to $(\Lpaynn)$ has a subsequence $(y^{n_k}, u^{n_k})$ converging weakly to a solution $(\bar{y}, \bar{u}$) of \eqref{prob:llpproblem}, and
\begin{equation*}
\lim_{k \to \infty} \Psi(u^{n_k}) = \Psi(\bar{u}).
\end{equation*}
\end{lem}
\begin{proof}
The proof is divided into three steps.

\begin{labeling}{Step 1:}

\item [Step 1:] We first aim at showing that the sequence $(y^n,u^n)$ has a weakly convergent subsequence in $\Fad$. Since $\Fad$ is a weakly closed subset of a reflexive Banach space, for this purpose it is sufficient to show that $(y^n,u^n)$ is bounded. Utilizing \ref{cond:boundsolop}, in turn, the boundedness of $(y^n, u^n)$ follows if we can prove that $(u^n)$ is bounded.

To show that $(u^n)$ is bounded, we argue as follows: Since the sequence $(\ydelta^n)$ is convergent, there exists $M> 0$ such that
\begin{equation*}
\|\ydeltaj^n \|_{\tilde{Y}} \leq M \quad \text{for all } n \in \N \text{ and } 1 \leq j \leq m.
\end{equation*}
A simple computation now shows that for every $(y,u) \in \Fad$ and every $n \in \N$ we have
\begin{equation*}
\ubar{\alpha} \cdot \Psi(u^n) \leq \Jayn(y^n,u^n) \leq \Jayn(y,u) \leq  \frac{1}{2}(\|y\|_{\tilde{Y}}+M)^2 +\bar{\alpha} \cdot \Psi(u).
\end{equation*}
Using that $\Psi$ is proper on $\Fad$, we can choose $(y,u) \in \Fad$ such that the right-hand side of this chain of inequalities is finite. Since the right-hand side is independent of $n$, and $\ubar{\alpha}>0$, this shows that
\begin{equation*}
\sum\limits_{i=1}^r \Psi_i(u^n)
\end{equation*}
is bounded. Consequently, from \ref{cond:coer} it follows that $(u^n)$ is bounded; and thus the first step is complete.

\item [Step 2:] Using the first step, we can assume that there exists a subsequence of $(y^n, u^n)$, which, for simplicity, we again denote by $(y^n, u^n)$, and $(\bar{y},\bar{u}) \in \Fad$ such that
\begin{equation*}
(y^{n},u^{n}) \rightharpoonup (\bar{y},\bar{u}).
\end{equation*}
Our goal in the second step is to show that $(\bar{y},\bar{u})$ solves \eqref{prob:llpproblem}. For this purpose, since $(y^n,u^n)$ solves $(\Lpaynn)$, note that
\begin{equation}\label{eq:nsolution}
\Jayn(y^n,u^n) \leq  \Jayn(y,u)
\end{equation}
for all $(y,u) \in \Fad$ and $n \in \N$. Using that
\begin{equation*}
(\alpha^n, \ydelta^n, y^n,u^n) \mapsto \Jayn(y^n,u^n)
\end{equation*}
is weakly lower semi continuous on $[\ubar{\alpha}, \bar{\alpha}] \times \tilde{Y}^m \times Y \times U$, and that for every $(y,u) \in \Fad$ the mapping
\begin{equation*}
(\alpha^n, \ydelta^n) \mapsto \Jayn(y,u)
\end{equation*}
is continuous on $[\ubar{\alpha}, \bar{\alpha}] \times \tilde{Y}^m$, taking the limit $n \to \infty$ in \eqref{eq:nsolution} we arrive at
\begin{equation}
\Jay(\bar{y},\bar{u}) \leq \liminf_{n \to \infty} \Jayn(y^n,u^n) \leq \lim_{n \to \infty} \Jayn(y,u) = \Jay(y,u).
\end{equation}

 As a consequence of this estimate, we have
\begin{equation}
\lim_{n \to \infty} \Jayn(y^n,u^n)=\Jay(\bar{y},\bar{u})  = \underset{(y,u) \in \Fad} {\min} \Jay(y,u) < \infty,
\end{equation}
which shows that $(\bar{y},\bar{u})$ solves \eqref{prob:llpproblem}. This finishes the second step.

\item [Step 3:] In order to complete the proof it remains to show that
\begin{equation*}
\lim_{n \to \infty} \Psi(u^n) = \Psi(\bar{u}),
\end{equation*}
which is done now. First, observe that due to weak lower semi continuity of the involved functions
\begin{equation}\label{eq:lsc1}
\| \bar{y} - \ydeltaj \|_{\tilde{Y}}^2 \leq \liminf_{n \to \infty}  \| y^n - \ydeltaj \|_{\tilde{Y}}^2 \quad \text{for } 1 \leq j \leq m
\end{equation}
and
\begin{equation}\label{eq:lsc2}
\Psi_i(\bar{u}) \leq  \liminf_{n \to \infty} \Psi_i(u^n) \quad \text{for } 1 \leq i \leq r.
\end{equation}
We now argue as follows: If
for some $1 \leq i \leq r$ it holds that
\begin{equation*}
\Psi_i(\bar{u}) <  \liminf_{n \to \infty} \Psi_i(u^n),
\end{equation*}
then in view of \eqref{eq:lsc1}--\eqref{eq:lsc2}, and using $\Jay(\bar{y},\bar{u})<\infty$, this implies
\begin{equation*}
\Jay(\bar{y},\bar{u}) < \lim_{n \to \infty} \Jayn(y^n,u^n),
\end{equation*}
Since we have already shown that
\begin{equation*}
\lim_{n \to \infty} \Jayn(y^n,u^n) = \Jay(\bar{y},\bar{u}),
\end{equation*}
this leads to a contradiction. Consequently, we must have
\begin{equation}\label{eq:liminfeq}
\Psi_i(\bar{u}) =  \liminf_{n \to \infty} \Psi_i(u^n) \quad \text{for all } 1 \leq i \leq r.
\end{equation}
Since \eqref{eq:liminfeq} is also true for every subsequence of $(u^n)$, this implies that
\begin{equation*}
\Psi_i(\bar{u}) =  \lim_{n \to \infty} \Psi_i(u^n) \quad \text{for all }  1 \leq i \leq r,
\end{equation*}
which is what was left to show.
\end{labeling}
\end{proof}

Strong convergence as in \Cref{defn:stabilitydata} and \ref{defn:stabilityparameters}, and thus stability,  can be achieved if the following additional assumptions are satisfied.
\begin{enumerate}[label=$(\text{H}\arabic*)$]
  \setcounter{enumi}{6}
\item For every sequence $(u^n)$ in $U$ and $u \in U$ it holds that, if
\begin{equation*}
u^n \rightharpoonup u \quad \text{and} \quad \Psi(u^n) \to \Psi(u),
\end{equation*}
then it follows that $u^n \to u$.\label{cond:weakstrong}
\item For each $u \in \Uad$ there exists a unique $y(u) \in Y$ such that

\begin{equation*}
e(y(u),u)=0,
\end{equation*}
and the mapping
\begin{equation*}
u \mapsto y(u)
\end{equation*}
is continuous from $\Uad$ to $Y$. \label{cond:solopcont}

\end{enumerate}

\begin{rem}
Condition \ref{cond:weakstrong} is known to hold, for instance, if
\begin{equation*}
\| \cdot \|_U = \sum\limits_{i=1}^r \Psi_i
\end{equation*}
and $U$ is a uniformly convex Banach space \cite[Proposition 3.32. on p.78]{brezis2010functional}.
\end{rem}

The following corollary, which under reasonable assumptions guarantees stability for the lower level problem, summarizes the considerations in this subsection.

\begin{cor}[Stability]\label{cor:stabilitytik} If \ref{cond:Uadccc}--\ref{cond:solopcont} hold, then \eqref{prob:llpproblem} is stable with respect to both the data and the regularization parameters.
\end{cor}
\begin{proof}

In combination with \ref{cond:weakstrong} and \ref{cond:solopcont} this is a direct consequence of \Cref{lem:stabilityLLP}.
\end{proof}

\subsection{Optimality conditions}

Optimality conditions for the lower level problem can be derived using standard Lagrangian methods. In this subsection we provide the main results needed for our purposes. Thereby we always make the following assumptions.
\begin{enumerate}[label=$(\text{A}\arabic*)$]
\item $e$ is well-defined and continuously F-differentiable on $Y \times U$. \label{cond:ediff}
\item $\Psi$ is continuously F-differentiable on $U$.
\end{enumerate}

\begin{defn}[First order necessary optimality conditions]
We say that a point $(y^\ast,u^\ast)$ satisfies the first order necessary optimality conditions of \eqref{prob:llpproblem}, if there exists $\lambda^\ast \in Z\dual$ such that
        
\begin{subequations}        
        \begin{align}   
y^\ast-\bar{\ydelta} + \lambda^\ast e_y(y^\ast,u^\ast) &=0 \label{eq:LLPadjoint}, \\
\langle \alpha \cdot \Psi_u(u^\ast) + \lambda^\ast e_u(y^\ast,u^\ast), u-u^\ast \rangle &\geq 0, \quad \forall u \in \Uad, \label{eq:LLPoptimality}\\
u^\ast \in \Uad, \quad e(y^\ast,u^\ast)&=0, \label{eq:LLPprimal}
\end{align}
\end{subequations}
where 
\begin{equation*}
\bar{\ydelta} \coloneqq \frac{1}{m} \sum\limits_{j=1}^m\ydeltaj.
\end{equation*}
Any point $(y^\ast,u^\ast, \lambda^\ast)$ such that \eqref{eq:LLPadjoint}--\eqref{eq:LLPprimal} hold, is called a KKT point  of \eqref{prob:llpproblem}.
\end{defn}

The following standard result is a special case of a theorem provided in $\cite{Zowe1979}$.

\begin{prop}\label{prop:firstordeLLP}
Let $(y^\ast,u^\ast)$ be a solution to \eqref{prob:llpproblem} such that $e_y(y^\ast,u^\ast)$ is bijective. Then there exists a unique $\lambda^\ast \in Z\dual$ such that $(y^\ast,u^\ast,\lambda^\ast)$ is a KKT point of \eqref{prob:llpproblem}. In particular, $(y^\ast,u^\ast)$ satisfies the first order necessary optimality conditions.
\end{prop}
\begin{rem}
If $U = \Uad$, then \eqref{eq:LLPadjoint}--\eqref{eq:LLPprimal} are equivalent to
\begin{subequations}  
\begin{align} 
y^\ast-\bar{\ydelta} + \lambda^\ast e_y(y^\ast,u^\ast) &=0, \label{eq:LLPadjoint2} \\
\alpha \cdot \Psi_u(u^\ast) + \lambda^\ast e_u(y^\ast,u^\ast) &=0,\label{eq:LLPoptimality2} \\
e(y^\ast,u^\ast)&=0.\label{eq:LLPprimal2}             
\end{align}
\end{subequations}
\end{rem}

\begin{defn}[Lagrange function]
We define the Lagrange function $\mathcal{L}_\alpha \colon Y \times \Uad \times Z\dual \to \R$ of the lower level problem by
\begin{equation*}
\mathcal{L}_\alpha(y,u,\lambda) \coloneqq \Jay(y,u) + \lambda e(y,u) \quad \text{for } (y,u,\lambda) \in Y \times \Uad \times Z\dual
\end{equation*}
for every $\alpha \in  [\ubar{\alpha}, \bar{\alpha}]$.
\end{defn}

\begin{defn}
We say that $(y^\ast,u^\ast)$ satisfies the second order sufficient optimality conditions of \eqref{prob:llpproblem}, if there exists $\lambda^\ast \in Z\dual$ and $\eta>0$ such that $(y^\ast,u^\ast,\lambda^\ast)$ is a KKT point and
\begin{equation*}\label{cond:secondordersufficient}
D^2_{(y,u)} \mathcal{L}_\alpha(y^\ast,u^\ast,\lambda^\ast)[(\delta_y,\delta_u),(\delta_y,\delta_u)]^2 \geq \eta \| (\delta_y,\delta_u) \|_{Y \times U}^2 \quad \text{for all }  (\delta_y,\delta_u) \in \ker De(y^\ast,u^\ast).
\end{equation*}
\end{defn}
The following result can be found in \cite{Maurer1979}.
\begin{prop} Let $(y^\ast,u^\ast)$ satisfy the second order sufficient optimality conditions of \eqref{prob:llpproblem}, and let $e_y(y^\ast,u^\ast)$ be bijective. Then $(y^\ast,u^\ast)$ is a local solution to \eqref{prob:llpproblem}.
\end{prop}

\section{Existence of solutions of the learning problem}\label{sec:exsol}
Using results from the previous section, we can now apply standard arguments to prove that the learning problem has a solution.

\begin{thm}\label{thm:genexbiopt}
If \ref{cond:Uadccc}--\ref{cond:weakly} hold, then \eqref{prob:genbiopt} has a solution.

\begin{proof}
We begin by showing that the feasible set of \eqref{prob:genbiopt}, which is given by
\begin{equation*}
\mathcal{F} \coloneqq \{(\alpha,y,u) \in [\ubar{\alpha}, \bar{\alpha}] \times \Fad \mid (y,u) \text{ solves } \eqref{prob:llpproblem} \},
\end{equation*}
is non empty and weakly sequentially compact. The non emptiness of $\mathcal{F}$ follows from \Cref{thm:llp}. In order to prove that $\mathcal{F}$ is weakly sequentially compact, we argue as follows: As a consequence of the Bolzano-Weierstraß theorem, every sequence $(\alpha^n,y^n,u^n)$ in $\mathcal{F}$ has a subsequence $(\alpha^{n_k},y^{n_k},u^{n_k})$  such that for some $\alpha^\ast \in [\ubar{\alpha}, \bar{\alpha}]$
\begin{equation}\label{eq:alphcon}
\alpha^{n_k} \to \alpha^\ast .
\end{equation}
Utilizing that \eqref{prob:llpproblem} is weakly stable with respect to the regularization parameters (\Cref{lem:stabilityLLP}) we can assume, possibly after taking another subsequence, that in addition to \eqref{eq:alphcon}
\begin{equation}
(y^{n_k},u^{n_k}) \rightharpoonup (y^\ast, u^\ast) 
\end{equation}
for some $(y^\ast,u^\ast) \in \Fad$ which solves $(\Lpayast)$. Since $(\alpha^\ast,y^\ast,u^\ast) \in \mathcal{F}$, this proves that $\mathcal{F}$ is weakly sequentially compact. In view of the fact that a weakly lower semi continuous function attains a minimum on a non empty and weakly sequentially compact set (see \eg \cite[Theorem 2.3 on p.8]{jahn2007introduction}), it remains to show that the mapping
\begin{equation*}
(\alpha,y,u) \mapsto \|u-\uex\|_{\tilde{U}}^2
\end{equation*}
is weakly lower semi continuous on $\mathcal{F}$. This follows from \cite[Corollary 3.9 on p.61]{brezis2010functional}, and thus the proof complete.
\end{proof}
\end{thm}

\section{Optimality conditions}\label{sec:optcond}
Throughout this section we make the following assumptions.
\begin{enumerate}[label=$(\text{B}\arabic*)$]
\item $e$ is well-defined and twice continuously F-differentiable on $Y \times U$. \label{cond:etwicediff}
\item $\Uad=U$, \ie there are no control constraints in the lower level problem. \label{cond:nocontrolconstraints}
\item $\Psi$ is twice continuously F-differentiable on $U$.
\item $e_y(y,u)$ is bijective for all $(y,u) \in Y \times U$. \label{cond:regLLP2}
\end{enumerate}
In a first step towards deriving optimality conditions for the learning problem, we consider its so-called KKT reformulation. In this reformulation, the lower level problem is replaced by its first order necessary optimality  conditions \eqref{eq:LLPadjoint2}--\eqref{eq:LLPprimal2}.
\begin{equation}\LeftEqNo\label{prob:genbioptrel}\tag{$\overline{\mathcal{LP}}$}
\left \{
\begin{aligned}
\underset{\alpha \in [\ubar{\alpha},\bar{\alpha}],\, (y,u,\lambda) \in Y \times U \times Z\dual} {\min}
\ \| u - \uex\|_{\tilde{U}}^2 \quad  \text{subject to} \\
 y-\bar{\ydelta} + \lambda e_y(y,u) &=0 \\
 \alpha \cdot \Psi_u(u) + \lambda e_u(y,u) &=0 \\
e(y,u)&=0.             \\
\end{aligned} \right.
\end{equation}
If the lower level problem is convex for every $\alpha \in [\ubar{\alpha}, \bar{\alpha}]$, then the learning problem \eqref{prob:genbiopt} and its KKT reformulation \eqref{prob:genbioptrel} are equivalent. In general, this is not the case since points which satisfy the necessary optimality conditions of the lower level problem are not necessarily solutions to the lower level problem. Before we address this issue, we note that at least for the KKT reformulation, optimality conditions can be derived by standard methods. In the following lemma, the assumption that the second order sufficient optimality condition holds is crucial, and serves as a constraint qualification.
\begin{lem}\label{opt:biopt} Let $(\alpha^\ast, y^\ast,u^\ast,\lambda^\ast)$ be a local solution to \eqref{prob:genbioptrel} with $(y^\ast,u^\ast, \lambda^\ast)$ satisfying the second order sufficient optimality condition of $(\Lpayast)$. Then there exists a unique $(p^\ast,q^\ast,z^\ast) \in Y \times U \times Z\dual$ such that
\begin{subequations}
\begin{align}
\langle \Psi_u(u^\ast) q^\ast, \alpha-\alpha^\ast \rangle_2 \geq 0, \quad \forall \alpha \in [\ubar{\alpha}, \bar{\alpha}],\label{eq:alpha} \\ 
p^\ast + \lambda^\ast e_{yy}(y^\ast,u^\ast) p^\ast +\lambda^\ast e_{yu}(y^\ast,u^\ast) q^\ast+z^\ast e_y(y^\ast,u^\ast)=0, \label{eq:state} \\
u^\ast-\uex + \lambda^\ast e_{uy}(y^\ast,u^\ast) p^\ast + \alpha^\ast \cdot \Psi_{uu}(u^\ast) q^\ast + \lambda^\ast e_{uu} (y^\ast,u^\ast) q^\ast + z^\ast e_u(y^\ast,u^\ast)=0, \label{eq:control}\\
e_y(y^\ast,u^\ast) p^\ast   + e_u(y^\ast,u^\ast) q^\ast =0. \label{eq:adjoint}
\end{align}
\end{subequations}
\end{lem}
\begin{proof}
A proof is given in the Appendix.
\end{proof}
We note that the equalities \eqref{eq:state}, \eqref{eq:control}, \eqref{eq:adjoint} hold in the spaces $Y\dual$, $U\dual$, $Z$, respectively, and typically represent partial differential equations. Since we want to use the optimality conditions from \Cref{opt:biopt} for the original learning problem, it is important to know when solutions to the learning problem are at least local solutions of its KKT reformulation. This is addressed in the following theorem, where it is implicitly assumed that the lower level problem \eqref{prob:llpproblem} admits a solution for every $\alpha \in [\ubar{\alpha}, \bar{\alpha}]$.

\begin{thm}\label{thm:localsolution}
Let $(\alpha^\ast,y^\ast,u^\ast)$ be a solution to \eqref{prob:genbiopt}. Assume that the following statements hold.
\begin{enumerate}[label=(\roman*)]
\item $(\Lpayast)$ is stable with respect to the regularization parameters.
\item $(y^\ast,u^\ast)$ satisfies the second order sufficient optimality condition of $(\Lpayast)$.
\item $(y^\ast,u^\ast)$ is the unique solution to $(\Lpayast)$.
\end{enumerate}
Then there exists a unique $\lambda^\ast \in Z\dual$ such that $(\alpha^\ast,y^\ast,u^\ast,\lambda^\ast)$ is a local solution to \eqref{prob:genbioptrel}.
\end{thm}
\begin{proof}
A proof is given in the Appendix.
\end{proof}
Note that by \Cref{cor:stabilitytik}, the first condition in \Cref{thm:localsolution} is satisfied, if \ref{cond:Uadccc}--\ref{cond:solopcont} hold . The second condition is also needed to ensure the existence of an optimality system for \eqref{prob:genbioptrel}. The third condition seems to be quite restrictive. Unfortunately, as indicated by a counterexample in \cite[Example 4.2.1]{HollerThesis2017}, without the third condition the conclusion of \Cref{thm:localsolution} no longer remains true.

\section{Examples}\label{sec:examples}

\subsection{Linear state equation}
We consider the class of problems
\begin{equation}\LeftEqNo\label{prob:genbioptlin}\tag{$\mathcal{LP}_{lin}$}
\begin{cases}
\underset{\alpha \in [\underline{\alpha},\bar{\alpha}],\, (y_\alpha,u_\alpha) \in Y \times U} {\min}
\|  u_\alpha - \uex \|_{\tilde{U}}^2 \qquad  \text{subject to}
        \\ \\
        (y_\alpha,u_\alpha) \in \argmin\limits_{\substack{u \in U \\ y \in Y}} \bigg\{\frac{1}{2m} \sum\limits_{j=1}^m  \|  y-\ydeltaj\|_{\tilde{Y}}^2 + \alpha \cdot \Psi(u) \mid e(y,u) = 0 \bigg\},
        \end{cases}
\end{equation}
with penalty functionals given by
\begin{equation*}
\Psi_i(u) = \frac{1}{2} \| K_i u \|_{E_i}^2 \quad \text{for } 1 \leq i \leq r,
\end{equation*}
and a linear state equation, \ie
\begin{equation*}\label{ass:linear}
e(y,u) = Ay - B u  \quad \text{for} \ (y,u) \in Y \times U,
\end{equation*}
where $A \in L(Y,Z)$, $B \in L(U,Z)$, and $K_i \in L(U,E_i)$ with $E_i$ being Hilbert spaces for $1\leq i \leq r$. The following assumptions are invoked to ensure that \eqref{prob:genbioptlin} has a solution, and that every solutions satisfies appropriate optimality conditions.

\begin{enumerate}[label=$(\text{L}\arabic*)$]
\item $A$ is bijective from $Y$ to $Z$. \label{cond:Abijcetive}

\item There exists $\zeta > 0$ such that
\begin{equation*}
\zeta \| u \|_U^2 \leq  \sum\limits_{i=1}^r \| K_i u \|_{E_i}^2  \quad \text{for all } u \in U.
\end{equation*}\label{cond:ellipticity}

\end{enumerate}

\begin{ex}
As a concrete example consider $Y=\Hen$, $U=\He$, $\tilde{Y}=\tilde{U}=E_i=\Ltwo$, $Z=\Hend$, where $\Omega \subset \R^d$, for $d \in \N$, is a bounded Lipschitz domain, let the state equation be given by
\begin{equation*}
e(y,u)= -\Delta y - u \quad \text{for } (y,u) \in \Hen \times \He,
\end{equation*}
and assume that weighted $H^1$-regularization is used, \ie
\begin{equation*}
K_i = \id \quad \text{and} \quad K_i=\partial_{x_{i-1}} \quad \text{for } i=2,\dots,d+1.
\end{equation*}
The invertibility required by \ref{cond:Abijcetive} can be derived using Poincaré's inequality and the Lax-Milgram lemma; see \eg \cite[Corollary 9.19 and Corollary 5.8]{brezis2010functional}.

\end{ex}

\paragraph{Existence of solutions} 

\begin{prop}\label{prop:exlinearstate}
If \ref{cond:Abijcetive}--\ref{cond:ellipticity} hold, then \eqref{prob:genbioptlin} has a solution.
\end{prop}
\begin{proof}
In view of \Cref{thm:genexbiopt} we only have to verify \ref{cond:Uadccc}--\ref{cond:weakly}, which can be done using standard arguments.
\end{proof}

\paragraph{Optimality conditions}
Since the lower level problem in \eqref{prob:genbioptlin} is strictly convex, the following optimality system can be easily derived from \Cref{opt:biopt} by making a few straightforward computations.
\begin{prop}\label{optcond:linstate}
Let $(\alpha^\ast,A^{-1}B u^\ast, u^\ast)$ be a solution to \eqref{prob:genbioptlin}. Then there exists $q^\ast \in U$ such that
\begin{subequations}
\begin{align}
\langle \Psi_u(u^\ast) q^\ast, \alpha-\alpha^\ast \rangle_2 \geq 0, \quad \forall \alpha \in [\ubar{\alpha}, \bar{\alpha}], \\
  u^\ast - \uex +  B^\ast A^{-\ast}  A^{-1} B q^\ast + \sum_{i=1}^r \alpha_i^\ast \mathcal{K}_i q^{\ast} = 0, \label{eq:optcondadjoint}  \\
\frac{1}{m} \sum_{j=1}^m  B^\ast A^{-\ast}  ( A^{-1} B u^\ast - \ydeltaj) + \sum_{i=1}^r \alpha_i^\ast \mathcal{K}_i u^\ast =0, \label{eq:optcondredLLP}
\end{align}
\end{subequations}
where
\begin{equation*}
\mathcal{K}_i \coloneqq K_i\adjoint K_i \quad \text{for } 1 \leq i \leq r.
\end{equation*}
\end{prop}

\begin{rem}
Note that \eqref{eq:optcondredLLP} is the optimality condition for the reduced lower level problem in the control variable, where we use that $(y,u) \in \Fad$ if and only if
\begin{equation*}
y= A^{-1}B u.
\end{equation*}
The adjoint equation for the learning problem is given by \eqref{eq:optcondadjoint}.
\end{rem}

\subsection{Bilinear state equation}\label{subsec:bilexample}

As an example with a bilinear state equation, we consider the estimation of the diffusion coefficient in a second order elliptic equation using $H^k$-regularization, where $k=1$ or $2$. This leads to the following problem.
\begin{equation}\LeftEqNo\label{prob:bileveldiff}\tag{$\mathcal{LP}_{bil}$}
\begin{cases}
\underset{\alpha \in [\underline{\alpha},\bar{\alpha}],\, (y_\alpha,u_\alpha) \in \Hen \times \Uad} {\min}
\|  u_\alpha - \uex \|_2^2 \qquad  \text{subject to}
        \\ \\
        (y_\alpha,u_\alpha) \in \argmin\limits_{\substack{u \in \Uad \\ y \in \Hen}} \{\frac{1}{2m} \sum\limits_{j=1}^m  \|  y-\ydeltaj\|_2^2 + \alpha \cdot \Psi(u) \mid e(y,u) = 0 \},
        \end{cases}
\end{equation}
with the state equation $e \colon \Hen \times \Uad \to \Hend$ given by
\begin{equation*}
e(y,u)= -\grad \cdot \, (  u \grad y) - f  \quad \text{for } (y,u) \in \Hen \times \Uad,
\end{equation*}
and the components of $\Psi$ given by 
\begin{equation*}
\Psi_\beta(u) = \frac{1}{2} \| \partial_\beta u  \|_{2}^2  \quad \text{for } \beta \in \N_0^d \text{ with }  |\beta| \leq k,
\end{equation*}
where $\Omega \subset \R^d$, for $d \in \N$, is a bounded Lipschitz domain,  $f \in \Ltwo$, and
\begin{equation*}
\Uad \coloneqq \{u \in \Hk \cap \Linfty \mid a \leq u \leq b \quad \text{a.e. in } \Omega \}
\end{equation*}
for $0<a<b<\infty$.

\paragraph{Existence of solutions}
The existence of solutions to \eqref{prob:bileveldiff} can be guaranteed without any additional assumptions.
\begin{prop}
The learning problem \eqref{prob:bileveldiff} has a solution.
\end{prop}
\begin{proof}
In view of \Cref{thm:genexbiopt}, it suffices to verify \ref{cond:Uadccc}--\ref{cond:weakly}. This can be done applying standard arguments. A proof for the case $k=1$ can be found in \cite[Proposition 5.4.2]{HollerThesis2017}. The same arguments as given there can be used for any $k \in \N$.
\end{proof}

\paragraph{Optimality conditions}
We aim at applying the results from \Cref{sec:optcond}, where we assume \ref{cond:etwicediff}--\ref{cond:regLLP2} to hold. To guarantee \ref{cond:etwicediff} for \eqref{prob:bileveldiff}, we require that $\Hk$ can be continuously embedded into $\Linfty$,
which is the case if and only if $k >d/2$; see \eg \cite[Theorem 5.4, Example 5.25, and 5.26]{adams1975sobolev}). Recall that the discussion in \Cref{sec:optcond} does not cover the case of control constraints. However, here control constraints are needed to ensure that \eqref{prob:bileveldiff} is well-posed. To circumvent this issue, we consider a relaxed version of \eqref{prob:bileveldiff}, in which $e$ is replaced by the relaxed state equation
\begin{equation*}
\tilde{e}(y,u)= -\grad \cdot \, (  \phi(u) \grad y) - f  \quad \text{for} \quad (y,u) \in \Hen \times \Hk,
\end{equation*}
where $\phi \colon \R \to \R$ is a (smoothed) pointwise projection onto $[a,b]$. The precise definition of $\phi$ is given in the Appendix. One can show that the learning problem with this relaxed state equation fulfills the assumptions of \Cref{thm:genexbiopt}, and thus has a solution. For $k > d/2$, the conditions \ref{cond:etwicediff}--\ref{cond:regLLP2} are satisfied. Consequently, in what follows we can use the results from \Cref{sec:optcond} to derive optimality conditions. Let us first state the KKT reformulation of the relaxed problem.
\begin{equation}\LeftEqNo\label{prob:relcontkkt}\tag{$\overline{\mathcal{LP}}_{rel}$}
\left \{
  \begin{aligned}
 \underset{(\alpha,y,u,p,) \in [\underline{\alpha},\bar{\alpha}] \times \Hen \times \Hk \times \Hen} {\min}
\ \|  u - \uex \|_2^2 \qquad  \text{subject to} \quad \\ 
 y-\bar{\ydelta} - \grad \cdot \, (\phi(u) \grad p) &= 0\\
\sum\limits_{\beta \in \N_0^d}^{|\beta|\leq k} \alpha_\beta \mathcal{K}_\beta u + \phi'(u) \grad y \cdot \grad p  &= 0 \\
 - \grad \cdot \, (\phi(u) \grad y ) - f &=0, \quad \end{aligned}
\right.
\end{equation}
where
\begin{equation*}
\mathcal{K}_\beta \coloneqq \partial_\beta\adjoint \partial_\beta \quad \text{for } \beta \in \N_0^d \text{ with }  |\beta| \leq k.
\end{equation*}
The following result can be seen as a straightforward consequence of \Cref{opt:biopt}.
\begin{prop}\label{prop:genbiopt}
Assume that $k > d/2$, and let $(\alpha^\ast,y^\ast,u^\ast,p^\ast)$ be a solution to \eqref{prob:relcontkkt}. If  $(y^\ast,u^\ast, p^\ast)$ satisfies the second order sufficient optimality conditions of $(\Lpayast)$ with the relaxed state equation, then there exists a unique $(q_1^\ast,q_2^\ast,q_3^\ast) \in \Hen \times \Hk \times \Hen$ such that
\begin{align*}
\langle \Psi_u(u^\ast) q^\ast, \alpha-\alpha^\ast \rangle_2 \geq 0, \quad \forall \alpha \in [\ubar{\alpha}, \bar{\alpha}], \\
q_1^\ast - \grad \cdot \, (\phi'(u^\ast)\, q_2^\ast \grad p^\ast) - \grad \cdot \, (\phi(u^\ast) \grad q_3^\ast) =0, \\
u^\ast-\uex + \phi'(u^\ast) \grad p^\ast \cdot \grad q_1^\ast + \sum\limits_{\beta \in \N_0^d}^{|\beta|\leq k} \alpha_\beta \mathcal{K}_\beta q_2^\ast + \phi''(u^\ast) q_2^\ast \grad y^\ast \cdot \grad p^\ast +  \phi^{'}(u^\ast) \grad y^\ast \cdot \grad q_3^\ast = 0, 
\\
-\grad \cdot \, ( \phi(u^\ast) \grad q_1^\ast )  -\grad \cdot\, (\phi'(u^\ast) q_2^\ast \grad y^\ast) =0.
\end{align*}

\end{prop}

\section{Numerical experiments}\label{sec:numexp}

In this section we present results for two numerical experiments regarding learning regularization parameters in weighted $H^1$-regularization.

\subsection{Linear state equation}

In the first experiment the inverse problem to be regularized is to estimate the forcing function in a second order elliptic partial differential equation.

\paragraph{Problem setting} We consider \eqref{prob:genbioptlin} with $\Omega=(-1,1) \times (-1,1)$, $Y=H^1_0(\Omega)$, $U=H^1(\Omega)$, $\tilde{Y}=\tilde{U}=E_i=\Ltwo$, and for $\gamma >0$ we define $e \colon \Hen \times \He \to \Hend$ by
\begin{equation*}
e(y,u) = -\gamma \Delta y+y-u \quad \text{for} \quad (y,u) \in \Hen \times \He.
\end{equation*}
We let the exact state $\yex$ be given as the solution to the control
\begin{equation*}
\uex(x_1,x_2)=\begin{cases}
2.5 \quad &\text{if }  | x_1-0.4 |<0.3\quad  \text{and} \quad | x_2-0.4 |<0.3, \\
2.5(\sin^2(2 \pi x_1) +{x_2}^2) \quad &\text{else}.
\end{cases}
\end{equation*}
The exact control is shown in \Cref{fig:utruelin}. We discretize the problem on a $128 \times 128$ mesh using the standard five-point stencil for the Laplace operator. Noisy data measurements $\ydeltaj$ are generated by pointwise setting
\begin{equation*}
\ydeltaj=\yex + \varepsilon \xi_j,
\end{equation*}
for $1 \leq j \leq m$, where $\xi_j$ follows a normal distribution with mean $0$ and standard deviation $1$, and $\varepsilon \coloneqq \epsilon \max | \yex| $ with $\epsilon$ being the relative noise level. We consider the following regularization operators.
\begin{equation*}
K_1\coloneqq \id,\quad  K_2\coloneqq \partial_{x_1}, \quad K_3\coloneqq \partial_{x_2}.
\end{equation*}
\begin{figure}

\begin{subfigure}[t]{0.5\textwidth}

\centering
 \captionsetup{width=.8\linewidth}
    \includegraphics[width=0.8\linewidth]{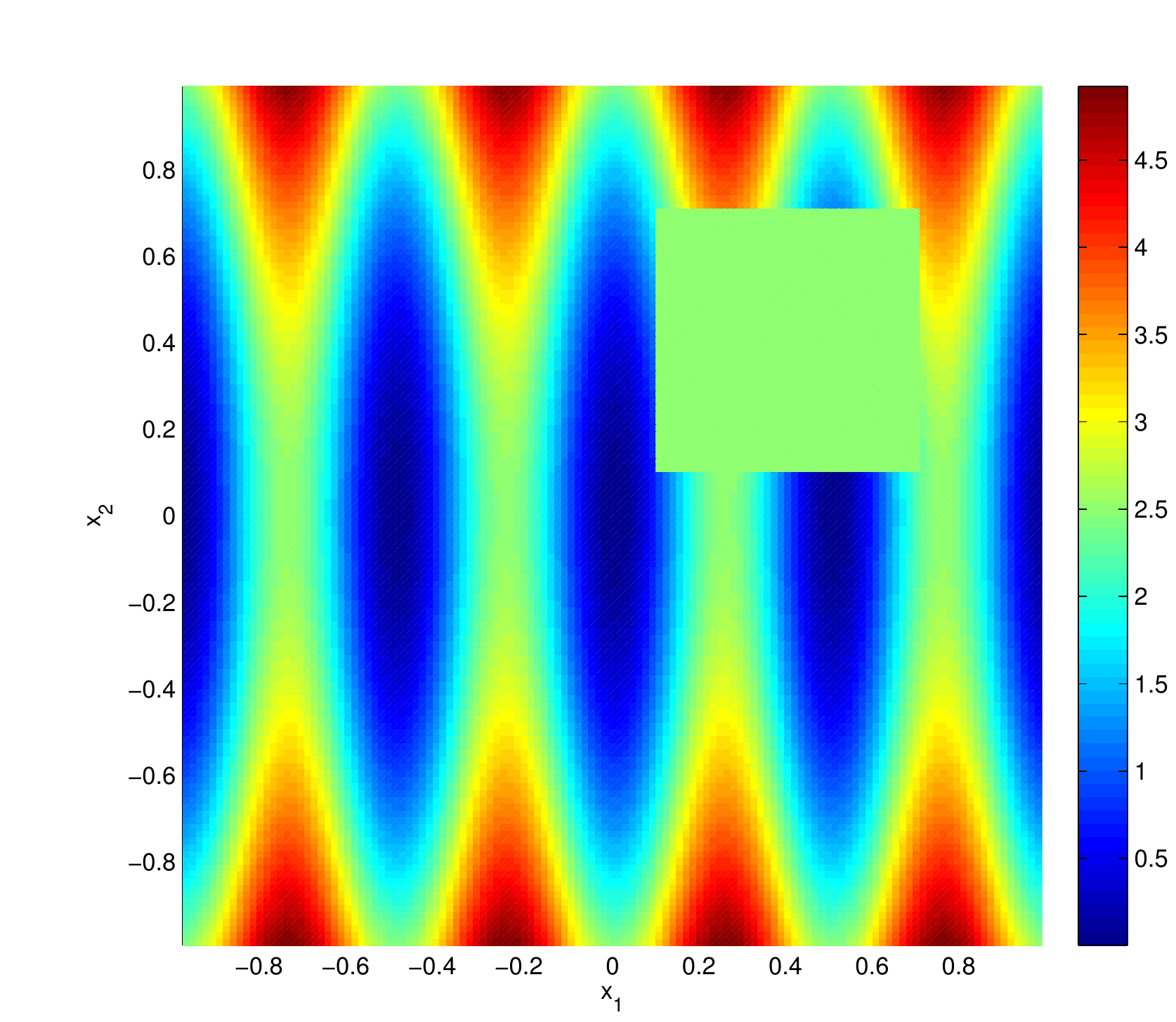}

    \subcaption{Control $\uex$ used to generate the exact state.}\label{fig:utruelin}
\end{subfigure}
\begin{subfigure}[t]{0.5\textwidth}

\centering
 \captionsetup{width=.8\linewidth}
    \includegraphics[width=0.8\linewidth]{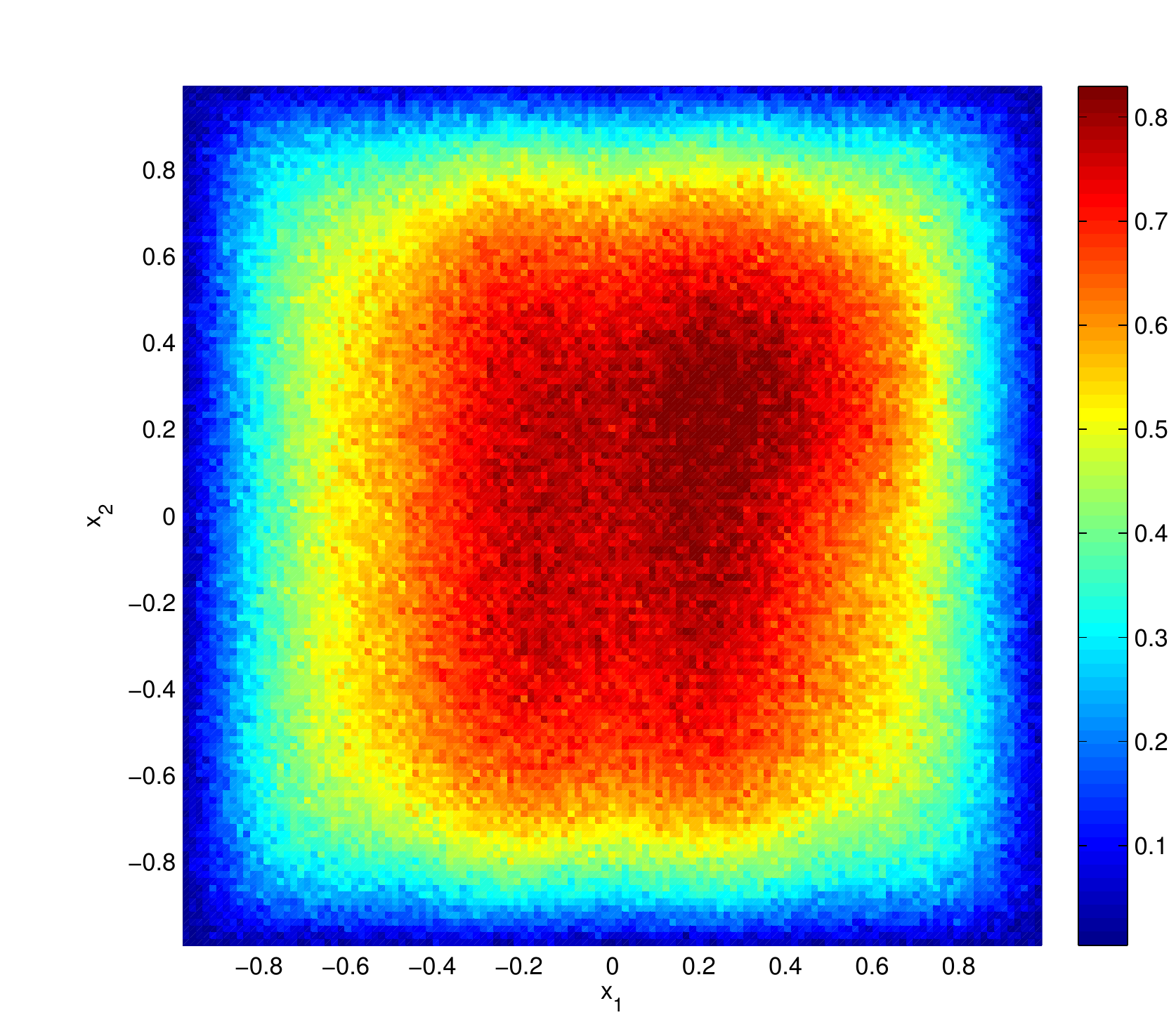}

\subcaption{Exact state with $10 \%$ noise added.}\label{fig:ynoiselin}
\end{subfigure}

\caption{Data used for the linear state equation.}

\end{figure}
In \Cref{fig:bilevelalpha}, we plot the values of the bilevel cost functional, \ie the squared distance between the recovered control and the exact control, in dependence of the regularization parameter when using the single operator $K_1$ for different noise levels. \Cref{fig:bilevelalpha2} shows the values of the bilevel cost functional in dependence of the regularization parameters using the operators $K_1$ and $K_2$ for $1 \%$ noise. Note that in both figures the bilevel cost functional seems to attain a distinct minimum. This motivates the feasibility of the formulation of finding regularization parameters as a learning problem. Additionally, the region in which the bilevel cost functional has non-negative curvature seems to be quite small.

\begin{figure}

\begin{subfigure}[t]{.5\textwidth}

\centering
 \captionsetup{width=.8\linewidth}
    \includegraphics[width=0.8\linewidth]{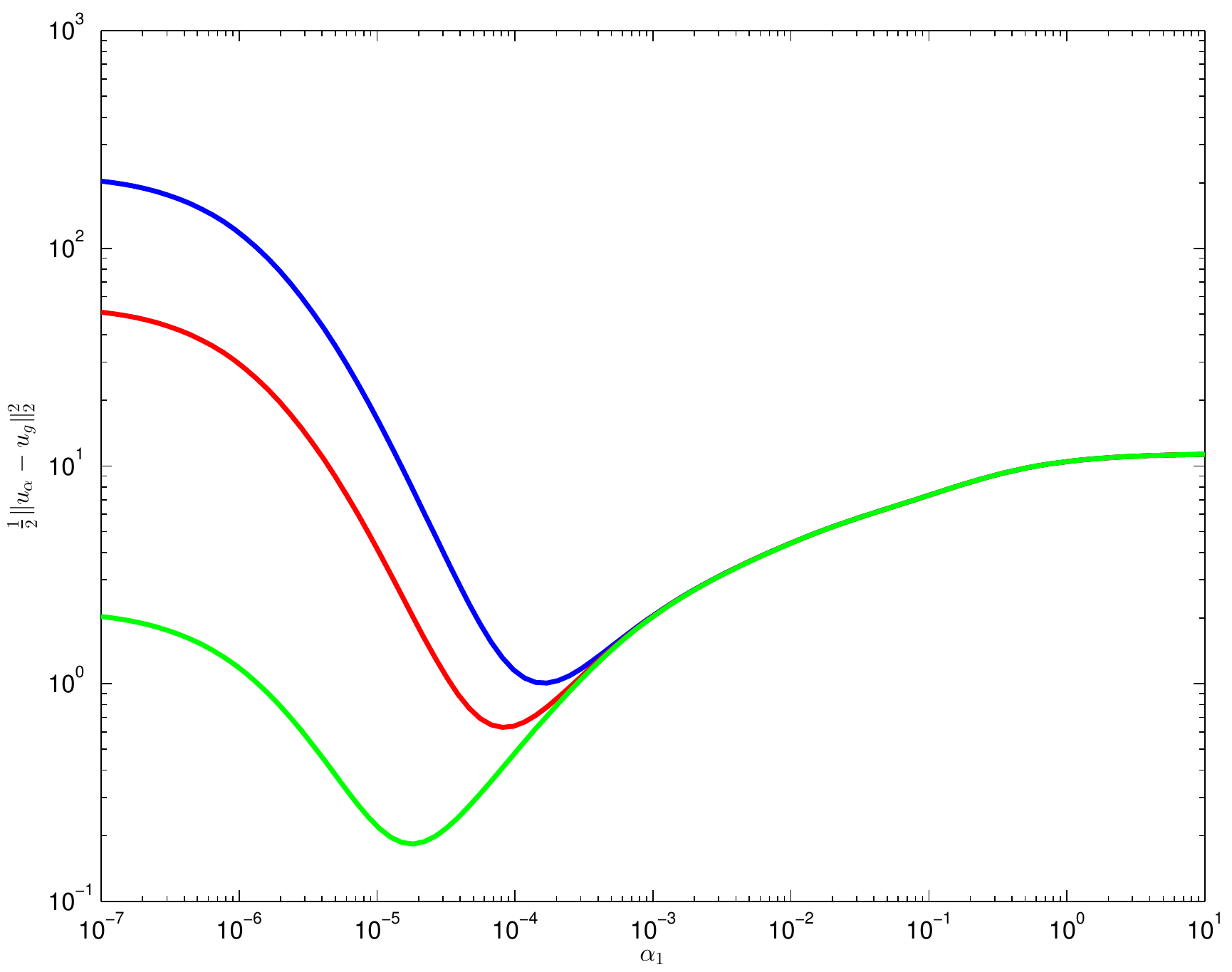}

\subcaption{Only using $K_1 = \id$ with $\gamma = 0.1$; using $10\%$
noise (blue), $5\%$ noise (red) and $1\%$ noise
(green).}   \label{fig:bilevelalpha}

\end{subfigure}
\begin{subfigure}[t]{.5\textwidth}

\centering
 \captionsetup{width=.8\linewidth}
    \includegraphics[width=0.8\linewidth]{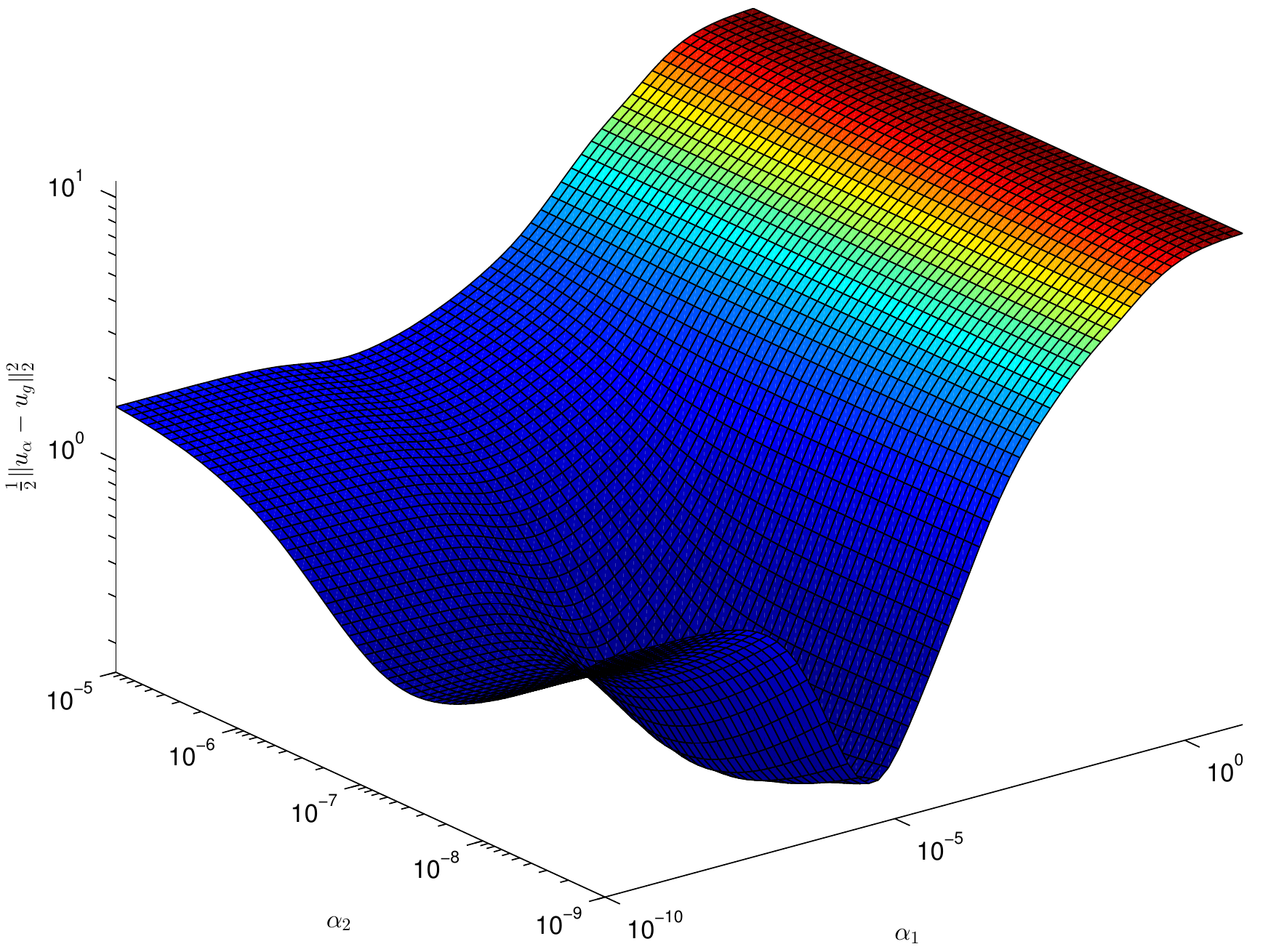}

\subcaption{Only using $K_1 = \id$ and $K_2=\partial_{x_1}$, with $\gamma = 0.1$  and $1\%$
noise. Here, we discretized the problem on a $64 \times 64$ mesh.}   \label{fig:bilevelalpha2}

\end{subfigure}
\caption{Values of the bilevel cost functional in dependence of the regularization parameters with a linear state equation in the lower level problem using one or two penalty functions. In both graphs we use a single noisy data measurement, \ie $m=1$.}
\end{figure}

\paragraph{Used methods and the solution algorithm}
To solve \eqref{prob:genbioptlin} we use a globalized quasi-Newton method. Since modifying the approximate Hessian to be positive definite would result in quite poor performance, we use a different strategy: We perform a regular BFGS update, unless we detect that a descent condition in the BFGS update direction is violated. In that case, instead, we perform a gradient descent update and reset the approximate Hessian (compare \cite[Algorithm 11.5 on p.60]{ulbrich2012nichtlineare}). In both cases, we perform an Armijo backtracking line search along the search directions. For a warm start, we always begin the iteration with $5$ initial gradient descent steps.
\begin{algorithm}
 \KwData{Let $\alpha^0$ be given.}
  Define $H^0 \coloneqq \id$. Compute $u^0$ solving the lower level problem for $\alpha^0$ and the corresponding Lagrange multiplier $q_0$ from the optimality system in \Cref{optcond:linstate}. For $1 \leq i \leq 3$, set $g^0_i \coloneqq \langle K_i q^0 ,K_i u^0 \rangle$, whenever the operator $K_i$ should be used, and set $d^0 \coloneqq-g^0$, $k\coloneqq 0$.
s

 \While{$\| g^k\|_2^2 > \text{tolerance}$}{

  \eIf{$\langle g^k,d^k\rangle < -\min\{c_1,c_2 \| d^k\|_2^2\} \|d^k\|_2^2$ and $k \geq 5$ }{
    Perform Armijo backtracking line search along $d^k$, set $k=k+1$ and update $\alpha^k$\;
   }{
   $H^k=\id$\;
   Perform Armijo backtracking line search along $-g^k$, set $k=k+1$ and update $\alpha^k$\;
  }
 Compute $u^k$ solving the lower level problem for $\alpha^k$  and the corresponding Lagrange multiplier $q^k$.
 Set $g^k_i \coloneqq \langle K_i q^k ,K_i u^k \rangle$, whenever the operator $K_i$ should be used. Update the approximate Hessian $H^k$ and compute the BFGS update direction $d^k$.}
 \caption{Iterative method for parameter learning with a linear state equation.}

\end{algorithm}
We terminated the algorithm, if the norm of the gradient fell below a certain threshold. In addition, for finer discretizations, we also terminated the algorithm if the Armijo backtracking line search was unsuccessful (which also indicates that we are close to a solution).

\paragraph{Results}
We tested the algorithm in MATLAB R2012b for various choices of operators $K_i$, for different noise levels as well as for a different number of available noisy data measurements.  To be able to compare results for the different settings, we used a fixed seed for random number generation for each noisy data measurement. We noticed the following behavior:
\begin{itemize}
\item  In all tested cases $K_1=\id$ is the best operator to use, if only one operator should be used. Using only $K_2=\partial_{x_1}$ or $K_3=\partial_{x_2}$ results in quite poor performances (see \Cref{fig:multipleoperators}, \Cref{table:single} and \ref{table:multiple} ).

\item Adding another regularization operator $K_i$ to any choice of one or two regularization operators improves tracking of the exact control (see \Cref{table:single} and \ref{table:multiple}).

\item Using $K_2=\partial_{x_1}$ and $K_3=\partial_{x_2}$ is the best choice amongst the two operator cases. The performance using these two operators is only slightly inferior to the performance using all three operators (see \Cref{table:single} and \ref{table:multiple}). This suggests that in this case the additional use of the regularization operator ${K_1}$ is not necessary.
\item When using multiple noisy data measurements $\ydeltaj$ with the same statistical structure, the ability to track $\uex$ is significantly improved, as we would expect (compare \Cref{table:single} with \Cref{table:multiple}).
\item When we only use unilateral regularization associated to $K_2$ or $K_3$, the optimal $u^\ast$ seems to have jumps in the direction which is not penalized (see \Cref{figg:2op}).

\end{itemize}

\begin{figure}
\begin{subfigure}[t]{0.5\textwidth}
\centering
\captionsetup{width=.8\linewidth}
    \includegraphics[width=0.8\linewidth]{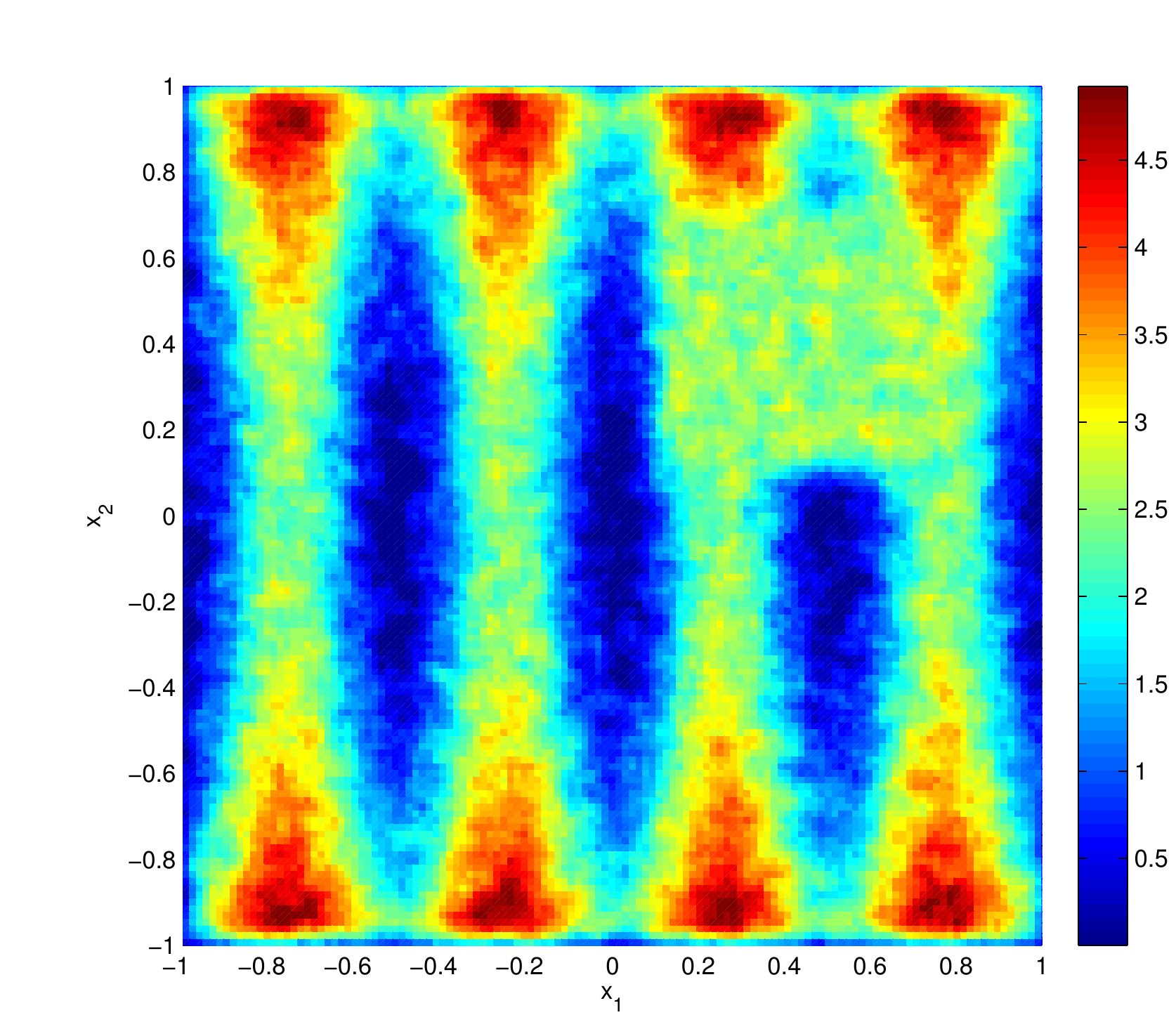}

    \caption{\protect\input{../fig/op1Data2Noise1m=1N=128}}
\end{subfigure}
\begin{subfigure}[t]{.5\textwidth}
\centering
\captionsetup{width=.8\linewidth}
    \includegraphics[width=0.8\linewidth]{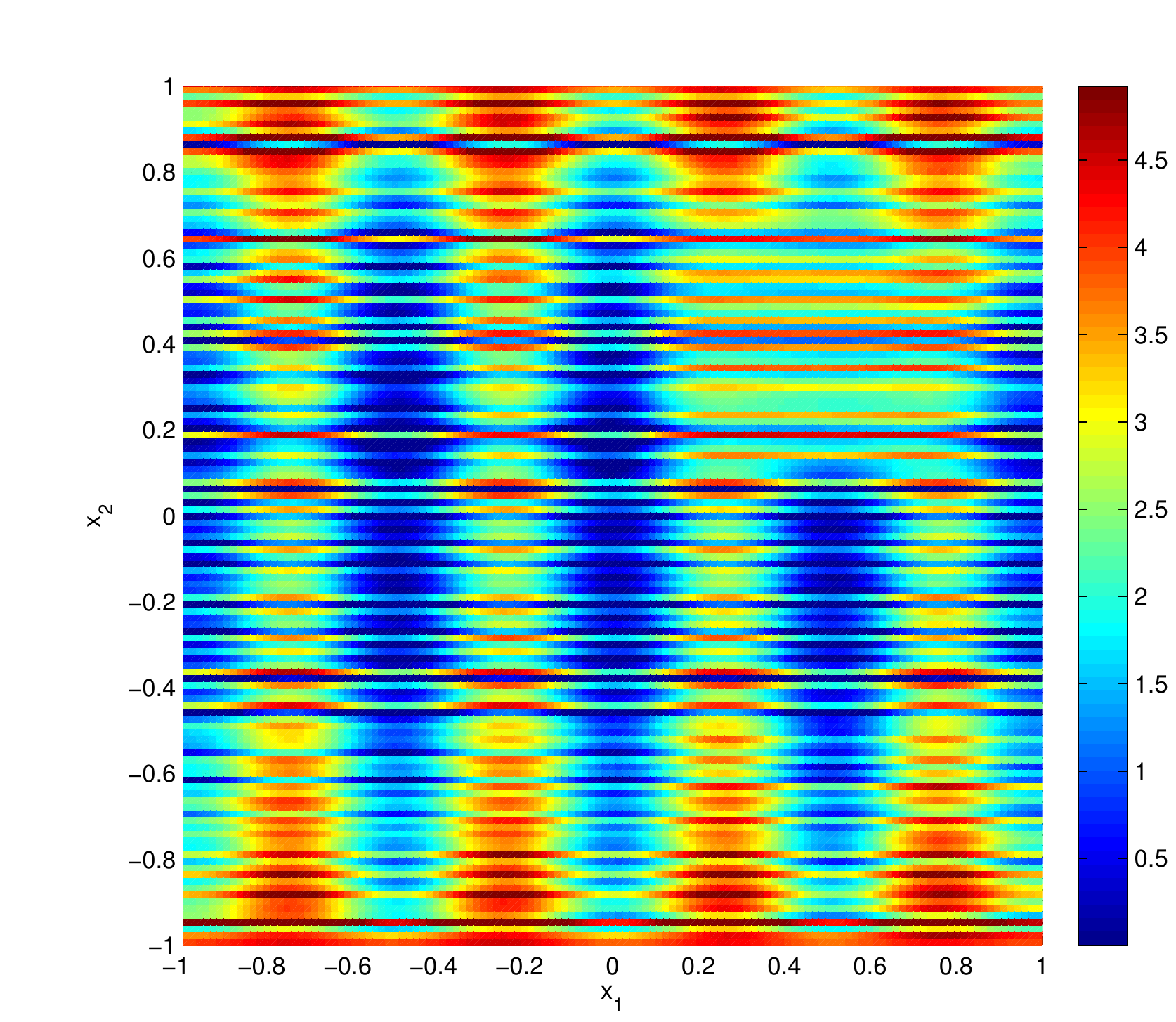}

    \caption{\protect\input{../fig/op2Data2Noise1m=1N=128}}\label{figg:2op}
\end{subfigure}

\caption{Optimal $u^\ast$ for the linear state equation, various choices of $K_i$, and $m=1$; $\gamma=0.1$ and $1 \%$ additive noise were used.}\label{fig:multipleoperators}
\end{figure}

\begin{figure}
\begin{subfigure}[t]{.5\textwidth}
\centering
\captionsetup{width=.8\linewidth}
    \includegraphics[width=0.8\linewidth]{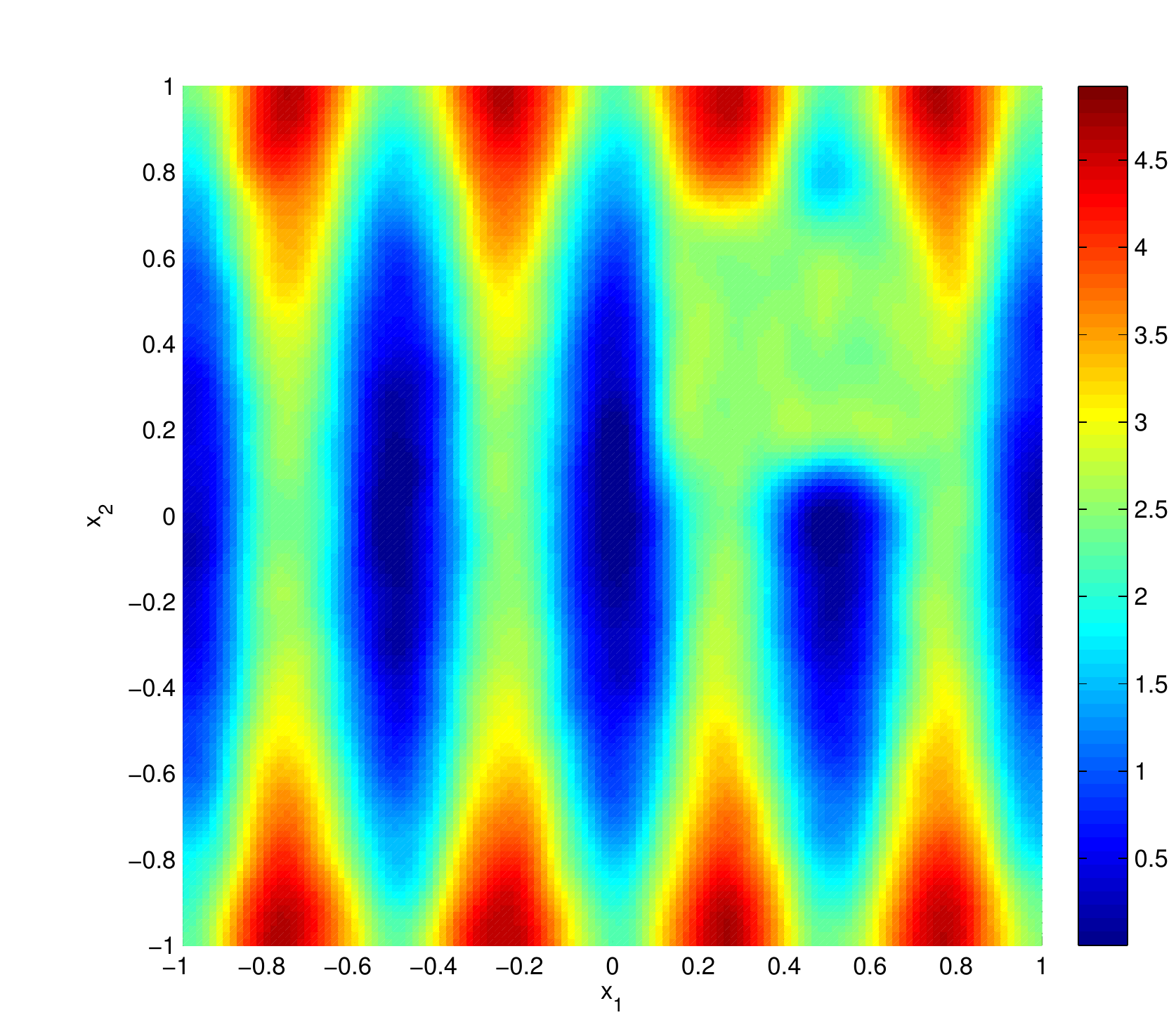}

    \caption{\protect\input{../fig/op23Data2Noise1m=1N=128}}

\end{subfigure}
\begin{subfigure}[t]{.5\textwidth}
\centering
\captionsetup{width=.8\linewidth}
    \includegraphics[width=0.8\linewidth]{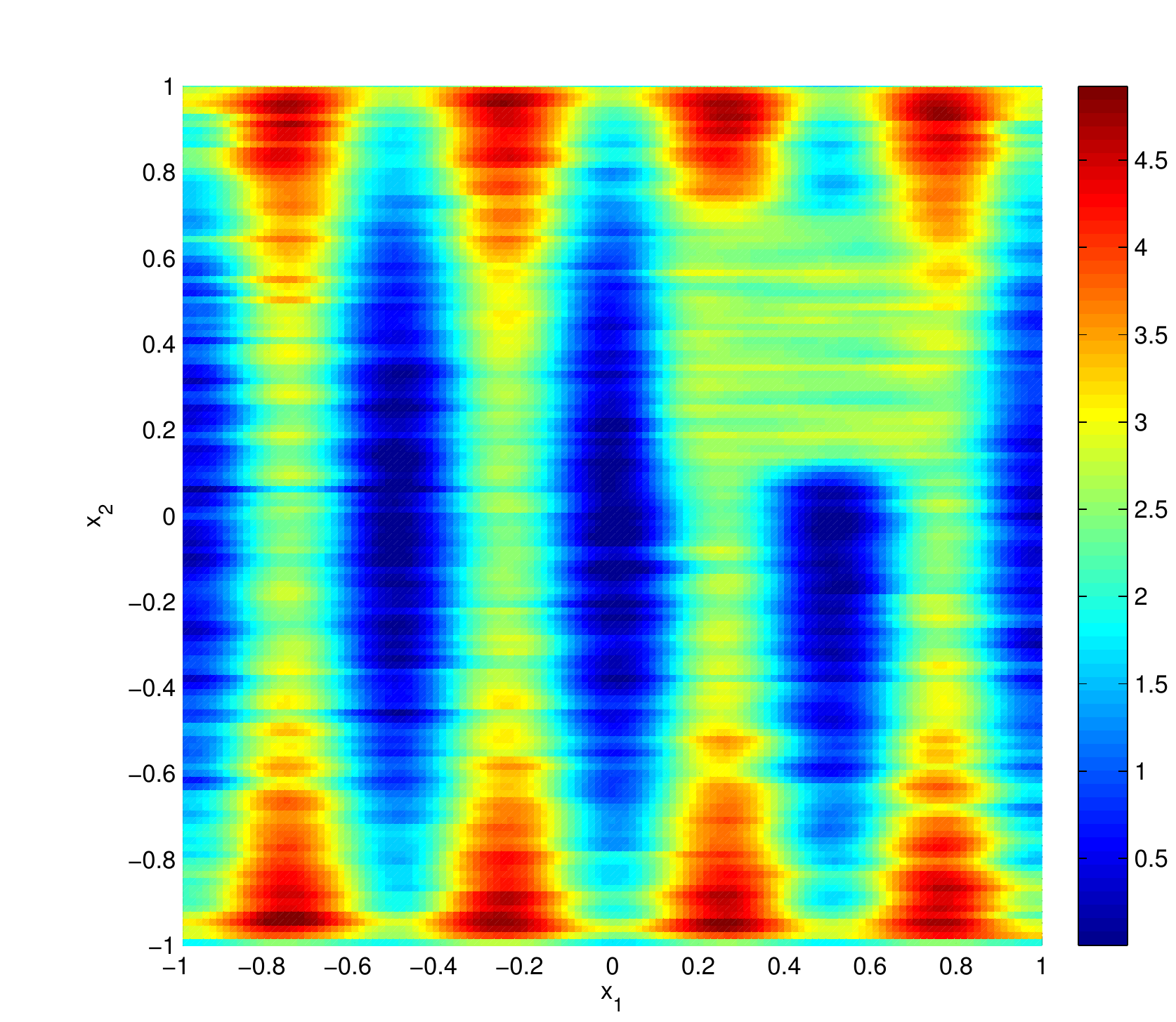}

    \caption{\protect\input{../fig/op12Data2Noise1m=1N=128}}

\end{subfigure}
\begin{subfigure}[t]{.5\textwidth}
\centering
\captionsetup{width=.8\linewidth}
    \includegraphics[width=0.8\linewidth]{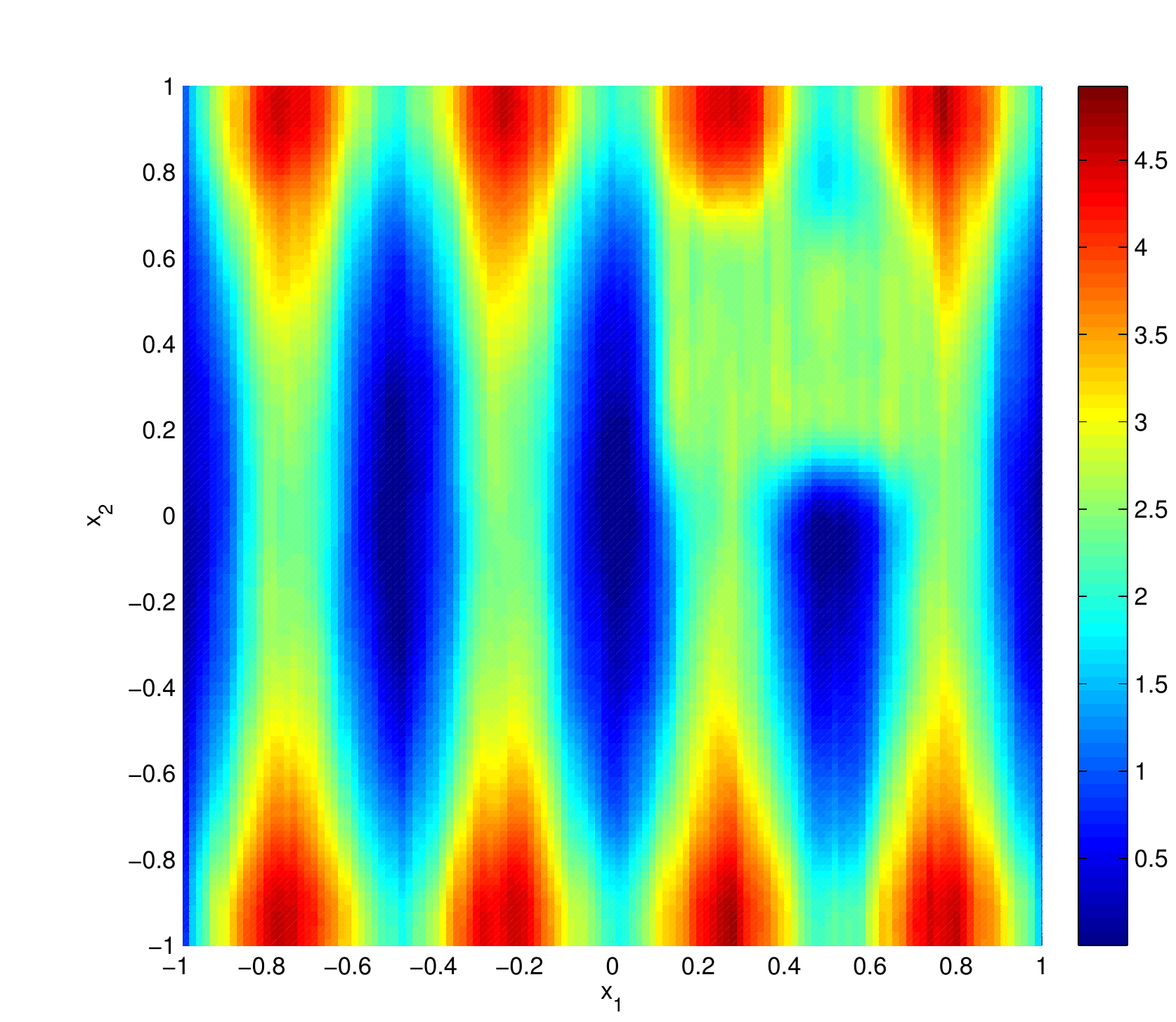}

    \caption{\protect\input{../fig/op13Data2Noise1m=1N=128}}

\end{subfigure}
\begin{subfigure}[t]{.5\textwidth}
\centering
\captionsetup{width=.8\linewidth}
    \includegraphics[width=0.8\linewidth]{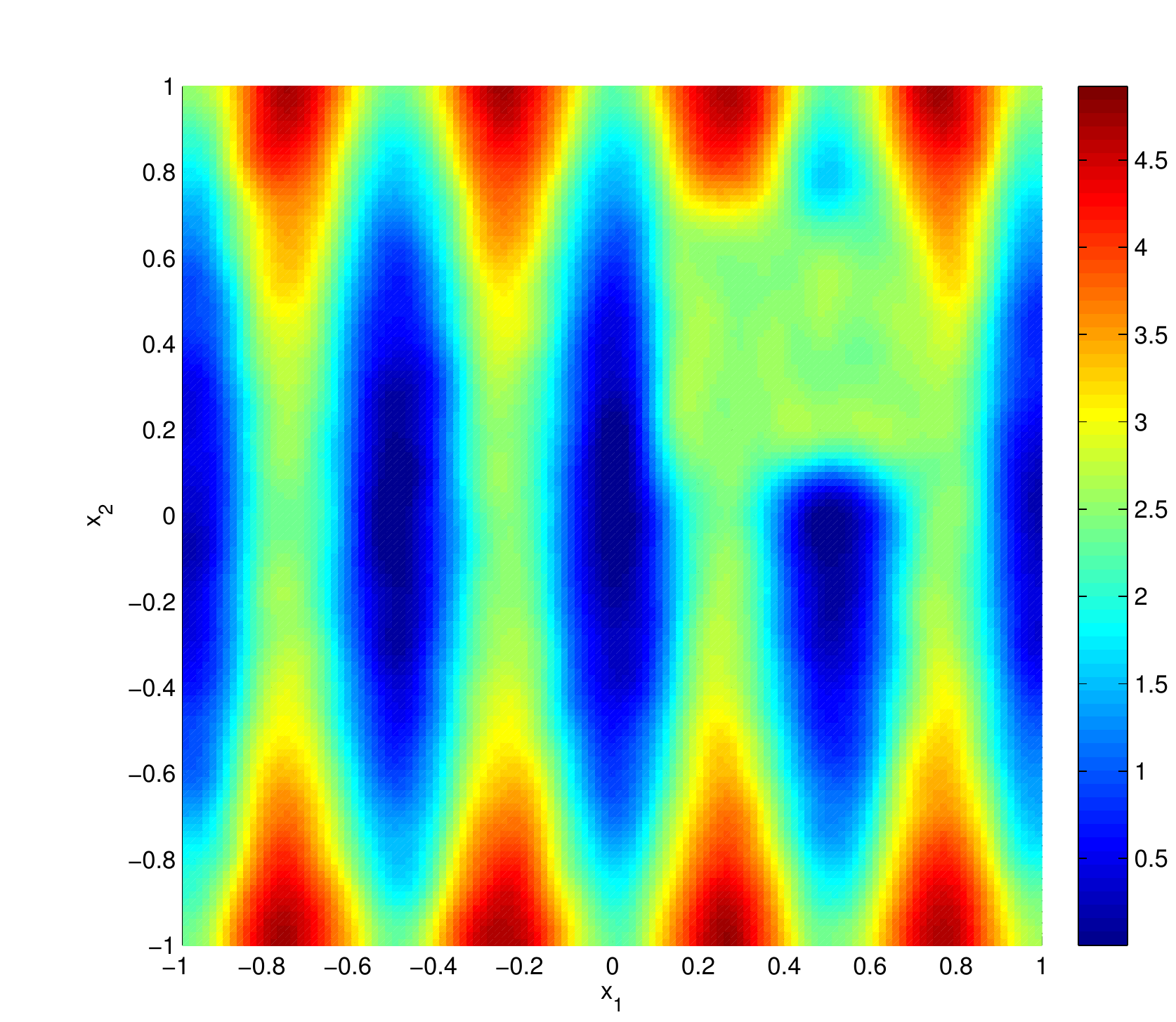}
\caption{\protect\input{../fig/op123Data2Noise1m=1N=128}}
\end{subfigure}

\caption{Optimal $u^\ast$ for the linear state equation, various choices of $K_i$, and $m=1$; $\gamma=0.1$ and $1 \%$ additive noise were used.}
\end{figure}

\begin{table}
\centering
    \begin{tabular}{| c | c | c |}
    \hline
    Used Operators & (Locally) Optimal $\alpha^\ast$ & $\|u^\ast-\uex\|_2^2$ \\ \hline
    \protect $K_1 $& $(\num{3.956409e-05 })$ & $1.6944$ \\ \hline
    \protect $K_2$& $(\num{2.857652e-07 })$ & $298.3233$ \\ \hline
    \protect $K_3$ & $(\num{1.371925e-06 })$ & $292.1254$ \\ \hline
    \protect $K_1 $, $K_2$& $(\num{1.214481e-05 },\num{3.781891e-08 })$ & $1.3169$ \\ \hline
    \protect $K_1 $, $K_3$ & $(\num{1.583561e-05 },\num{5.632605e-07 })$ & $0.35062$ \\ \hline
    \protect $K_2$, $K_3$ & $(\num{1.266485e-08 },\num{2.032731e-07 })$ & $0.20995$ \\ \hline
    \protect $K_1 $, $K_2$, $K_3$ & $(\num{7.043418e-07 },\num{1.252633e-08 },\num{2.042817e-07 })$ & $0.20972$ \\ \hline

    \end{tabular}
     \caption{Locally optimal $\alpha^\ast$ for different sets of operators $K_i$ with $\gamma = 0.1$, $10\%$ noise and $m=1$.}\label{table:single}

\end{table}
\begin{table}
\centering
    \begin{tabular}{| c | c | c |}
    \hline
    Used Operators & (Locally) Optimal $\alpha^\ast$ & $\|u^\ast-\uex\|_2^2$ \\ \hline
    \protect $K_1 $& $(\num{1.746421e-05 })$ & $1.2802$ \\ \hline
    \protect $K_2$& $(\num{1.042513e-07 })$ & $66.6712$ \\ \hline
    \protect $K_3$ & $(\num{4.502205e-07 })$ & $68.41$ \\ \hline
    \protect $K_1 $, $K_2$& $(\num{3.914312e-06 },\num{2.153651e-08 })$ & $0.89223$ \\ \hline
    \protect $K_1 $, $K_3$ & $(\num{6.171296e-06 },\num{2.067933e-07 })$ & $0.21795$ \\ \hline
    \protect $K_2$, $K_3$ & $(\num{4.986058e-09 },\num{7.343026e-08 })$ & $0.11918$ \\ \hline
    \protect $K_1 $, $K_2$, $K_3$ & $(\num{-5.430532e-07 },\num{5.091390e-09 },\num{7.222112e-08 })$ & $0.11881$ \\ \hline

    \end{tabular}
     \caption{Locally optimal $\alpha^\ast$ for different sets of operators $K_i$ with $\gamma = 0.1$, $10\%$ noise and $m=5$.}\label{table:multiple}

\end{table}

\FloatBarrier

\subsection{Bilinear state equation}
In the second numerical experiment the inverse problem is to estimate the diffusion coefficient in a second order elliptic partial differential equation. Note that since we use $H^1$-regularization for dimension $d=2$, the optimality system used to compute optimal regularization parameters is only obtained by formal computations.

\paragraph{Problem setting} We consider \eqref{prob:bileveldiff} with $\Omega=(-1,1) \times (-1,1)$, $Y=H^1_0(\Omega)$, $U=H^1(\Omega)\cap \Linfty$, $\tilde{Y}=\tilde{U}=L^2(\Omega)$, and let $e \colon \Hen \times U \to \Hend$ be given by
\begin{equation*}
e(y,u)=\grad \cdot \, (\phi(u) \grad y ) -f \quad \text{for } (y,u) \in \Hen \times U,
\end{equation*}
Recall that $\phi \colon \R \to \R$ was introduced \Cref{subsec:bilexample} to avoid control constraints. In the state equation, we choose $f \in \Ltwo$ such that for the control $\uex$ given by
\begin{equation*}
\uex(x_1,x_2)=\begin{cases}
1+{x_2}^2 \quad &\text{if }  \sqrt{{x_1}^2+{x_2}^2} \leq \frac{1}{2}, \\
0.1+{x_1}^2 \quad &\text{else},
\end{cases}
\end{equation*}
the exact state $\yex$ is given by
\begin{equation*}
\yex(x_1,x_2)=({x_1}^4-{x_1}^2)({x_2}^2-1).
\end{equation*}
The exact control is shown in \Cref{figbil:utruelin}. We discretize the problem on a $64 \times 64$ mesh using Lagrange $P_1$ finite elements. Noisy data measurements $\ydeltaj$ are generated by pointwise setting
\begin{equation*}
\ydeltaj=\yex + \varepsilon \xi_j,
\end{equation*}
for $1 \leq j \leq m$, where $\xi_j$ follows a normal distribution with mean $0$ and standard deviation $1$, and $\varepsilon = \epsilon \max | \yex| $ with $\epsilon$ being the relative noise level. We consider the following regularization operators
\begin{equation*}
K_1\coloneqq \id,\quad  K_2\coloneqq \partial_{x_1}, \quad K_3\coloneqq \partial_{x_2}.
\end{equation*}

\paragraph{Used methods and the solution algorithm}
We used nearly the same globalized quasi-Newton method as for the linear state equation. The only significant difference is that here a solution to the lower level problem is computed using the sequential programming method (SP method for short) from \cite{kunischsequential}.
\begin{algorithm}
 \KwData{Let $\alpha^0$ be given.}
 Define $H^0 \coloneqq \id$. Compute $u^0$ solving the lower level problem for $\alpha^0$ using the SP method, and the corresponding Lagrange multiplier $p^0$ and $(q_1^0,q_2^0,q_3^0)$ as in \Cref{prop:genbiopt}. For $1 \leq i \leq 3$ set $g_i^0 \coloneqq \langle K_i q_2^0 ,K_i u^0 \rangle$, whenever the operator $K_i$ is used, and set $d^0\coloneqq-g^0$, $k\coloneqq0$.

 \While{$\| g^k\|_2^2 > \text{tolerance}$}{

  \eIf{$\langle g^k,d^k\rangle < -\min\{c_1,c_2 \| d^k\|_2^2\} \|d^k\|_2^2$ and $k \geq 5$ }{
    Perform Armijo backtracking line search along $d^k$, set $k=k+1$ and update $\alpha^k$\;
   }{
   $H^k=\id$\;
   Perform Armijo backtracking line search along $-g^k$, set $k=k+1$ and update $\alpha^k$\;
  }
Compute $u^k$ solving the lower level problem for $\alpha^k$ using the SP method, and the corresponding Lagrange multiplier $p^k$ and $(q_1^k,q_2^k,q_3^k)$ as in \Cref{prop:genbiopt}. For $1 \leq i \leq 3$ set $g_i^k \coloneqq \langle K_i q_2^k ,K_i u^k \rangle$, whenever the operator $K_i$ is used.  Update the approximate Hessian $H^k$ and compute the BFGS update direction $d^k$.}
 \caption{Iterative method for \eqref{prob:bileveldiff}}

\end{algorithm}
We terminated the algorithm, if the norm of the gradient fell below a certain threshold. In addition, for finer discretizations, we also terminated the algorithm if the Armijo backtracking line search was unsuccessful (which also indicates that we are close to a solution).

\paragraph{Results}
We tested the algorithm in MATLAB R2012b for various choices of operators $K_i$, for different noise levels as well as for a different number of available noisy data measurements. As for the linear state equation, we used a fixed seed for random number generation for each noisy data measurement. We noticed the following behaviour:
\begin{itemize}
\item  $K_1=\id$ was the only choice of a single operator which lead to meaningful results (see \Cref{fig2:jump1}).

\item Adding another regularization operator $K_i$ to any choice of one or two regularization operators generally improves tracking of the exact control (see \Cref{tablebil:multiple5}). There is an outlier to this claim comparing the use of a single operator $K_1$ to the use of the two regularization operators $K_1$ and $K_3$ for $m=20$ (see \Cref{tablebil:multiple20}).

\item Using $K_2=\partial_{x_1}$ and $K_3=\partial_{x_2}$ is the best choice amongst the two operator cases. The performance using these two operators is only slightly inferior to using all three operators (see  \Cref{tablebil:multiple5}). This suggests, that in this case the additional use of the regularization operator ${K_1}$ is not necessary.
\item When using multiple noisy data measurements $\ydeltaj$ with the same statistical structure, the ability to track $\uex$ is generally improved, as we would expect (compare \Cref{tablebil:multiple5} and \ref{tablebil:multiple20}). Note that this is not the case comparing $m=5$ to $m=20$ when using the operators $K_1$ and $K_3$, but since the data was generated using a random process, this does not contradict the theory.

\item Using the operator $K_2=\partial_{x_1}$ seems to be more significant for the quality of the reconstructions than using $K_3=\partial_{x_2}$ (see \Cref{tablebil:multiple5} and \ref{tablebil:multiple20}). This is also indicated by observing that the obtained optimal regularization parameter for $K_2$ is usually larger than the optimal regularization parameter for $K_3$ when using both operators.

\item When we only use unilateral regularization associated to $K_2=\partial_{x_1}$ or $K_3=\partial_{x_2}$ together with $L^2$-regularization, the optimal $u^\ast$ usually suffers from over-smoothing in the penalized direction. In contrast,  $u^\ast$ can have rapid changes in the direction which is not penalized.

\item We expect difficulties reconstructing $\uex$ at stationary points of $\yex$ (see \cite[p.24]{engl1996regularization}). A simple computation shows that $(x_1,x_2)$ is a stationary point of $\yex$ if and only if one of the following statements is true:
\begin{enumerate}[label=\alph*)]
\item $x_1=0$ (line segment along the $x_2$-axis)
\item $|x_1|=1$ and $|x_2|=1$ (edges of the domain)
\item $|x_1|=\sqrt{1/2}$ and $x_2 = 0$
\end{enumerate}
Here we have continuously extended the gradient of $\yex$ to the boundary of the domain. Difficulties reconstructing $\uex$ near the edges of the domain can be seen in  \Cref{fig2:jump1}. Since in this case there is no additional smoothing in any of the directions, the values of the reconstructed $u^\ast$ near the edges tend to zero. Difficulties reconstructing $\uex$ near the $x_2$-axis can be seen in  \Cref{fig2:jump1} and \ref{fig2:jump13}. Note that smoothing in the $x_1$-direction, however, largely prevents the issues near the $x_2$-axis, as we can see in \Cref{fig2:jump23}.

\end{itemize}

\afterpage{\clearpage}
\begin{figure}

\begin{subfigure}[t]{0.5\textwidth}

\centering
\captionsetup{width=0.8\linewidth}
    \includegraphics[width=0.8\linewidth]{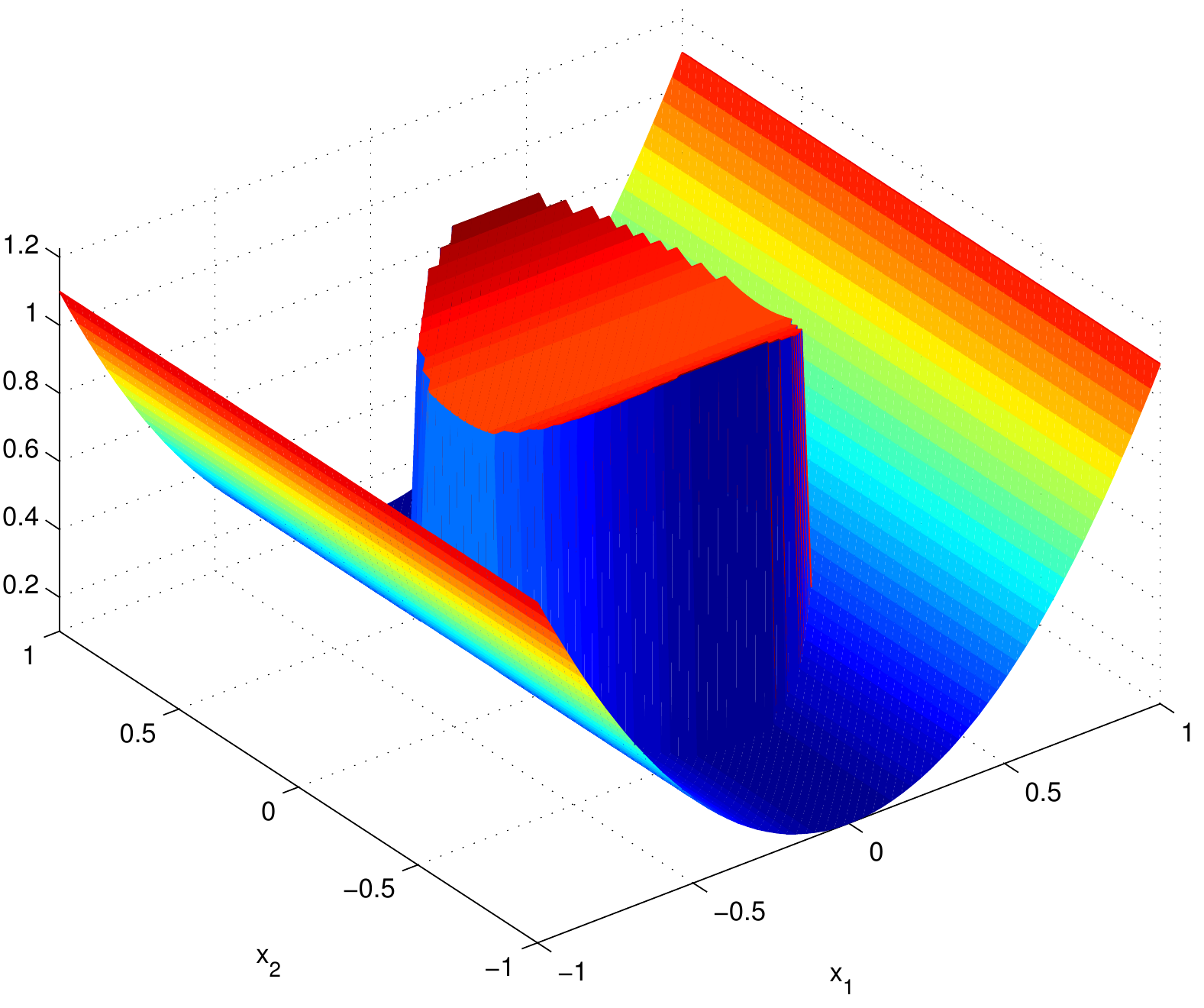}

    \subcaption{Control $\uex$ used to generate the exact state.}\label{figbil:utruelin}
\end{subfigure}
 \begin{subfigure}[t]{.5\textwidth}

\centering
\captionsetup{width=0.8\linewidth}
    \includegraphics[width=0.8\linewidth]{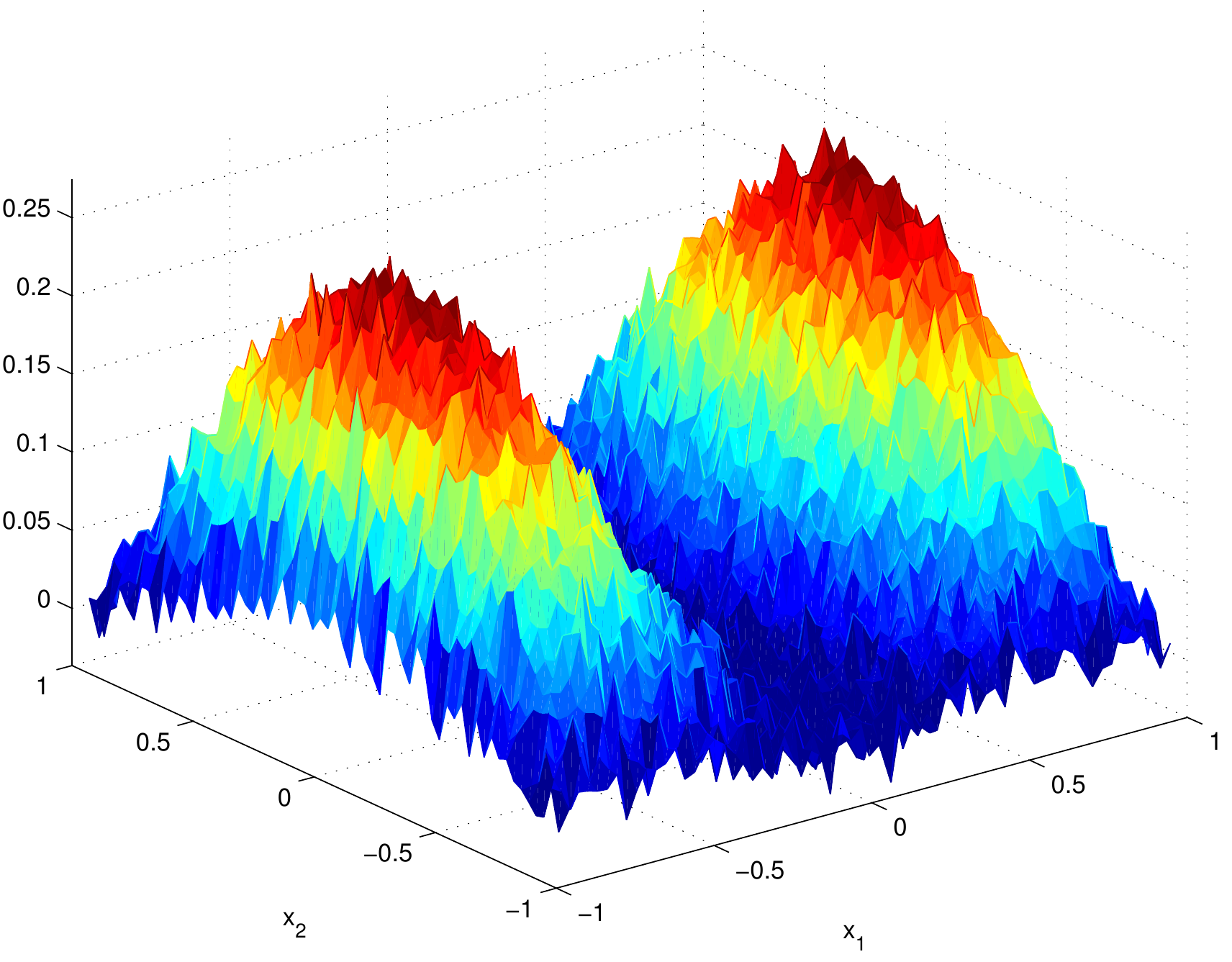}

\subcaption{Exact state with $5 \%$ noise added.}\label{figbil:ynoiselin}
\end{subfigure}
\caption{Data used for the bilinear state equation.}
\end{figure}

\begin{figure}

\begin{subfigure}[t]{.5\textwidth}
\centering
\captionsetup{width=0.8\linewidth}
    \includegraphics[width=0.8\linewidth]{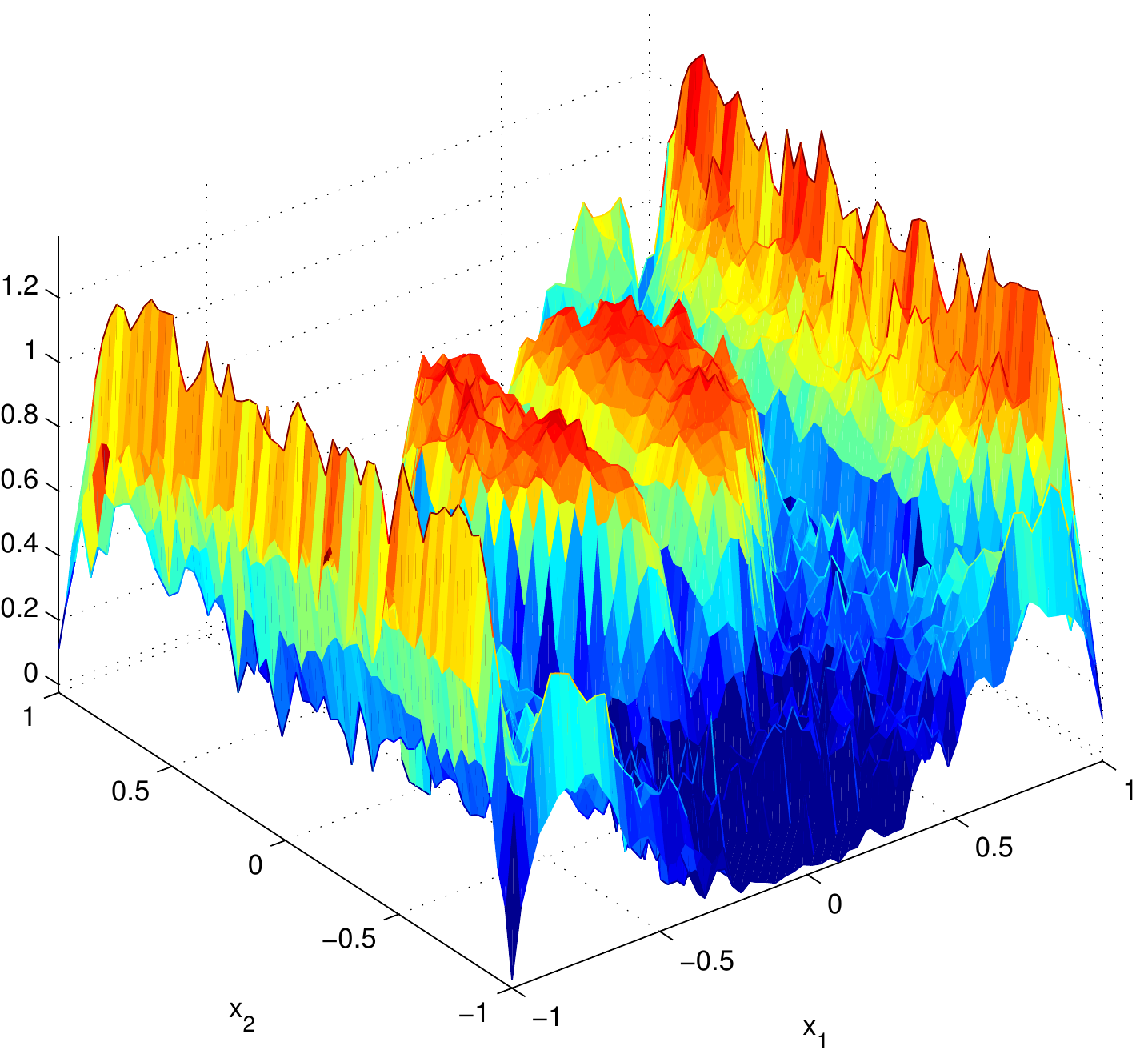}

    \caption{\protect\input{../figbil/op1BilData3Noise1m=1N=66}}\label{fig2:jump1}
\end{subfigure}
\begin{subfigure}[t]{.5\textwidth}
\centering
\captionsetup{width=0.8\linewidth}
    \includegraphics[width=0.8\linewidth]{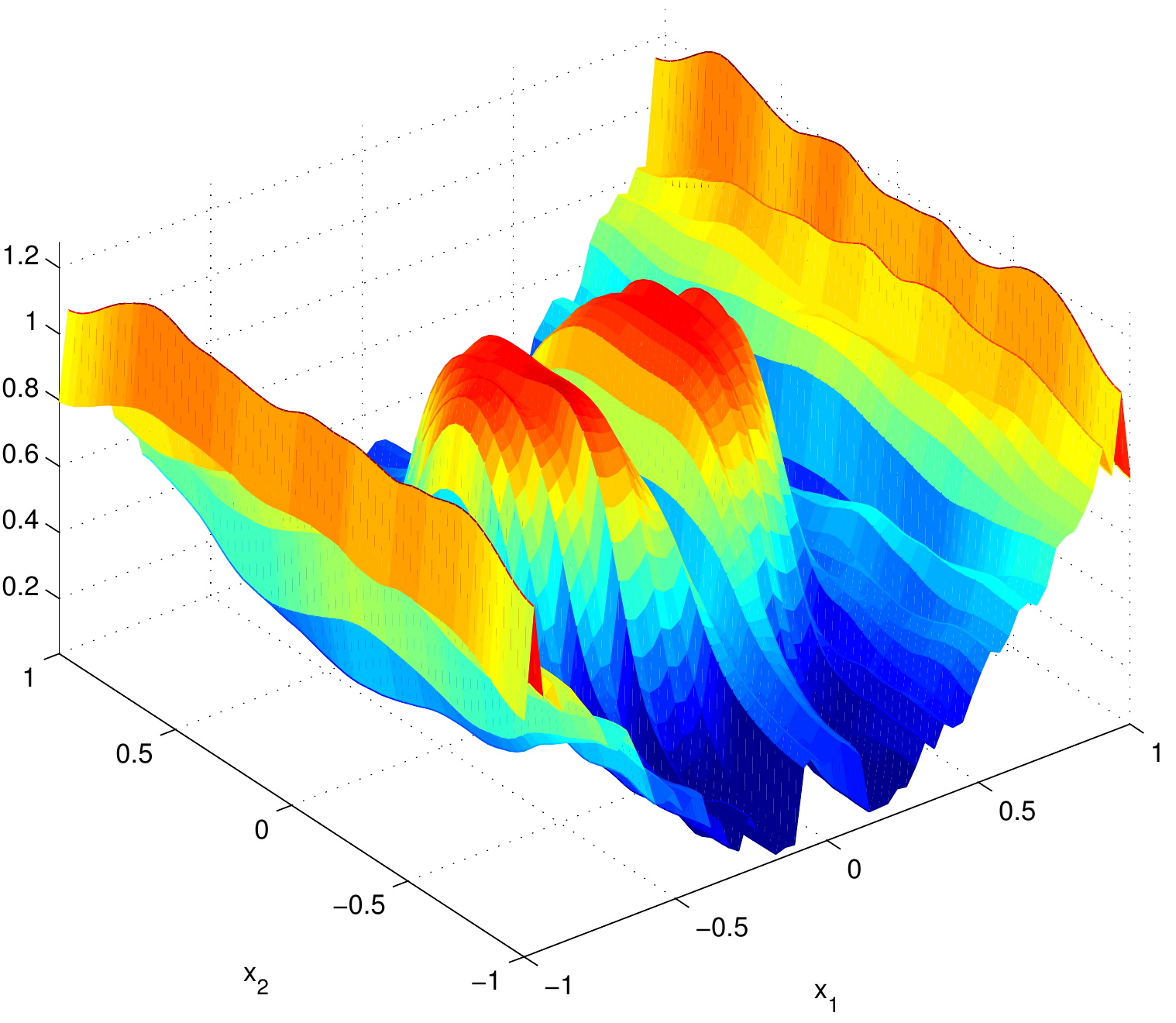}

    \caption{\protect\input{../figbil/op13BilData3Noise1m=1N=66}}\label{fig2:jump13}

\end{subfigure}
\begin{subfigure}[t]{.5\textwidth}
\centering
\captionsetup{width=0.8\linewidth}
    \includegraphics[width=0.8\linewidth]{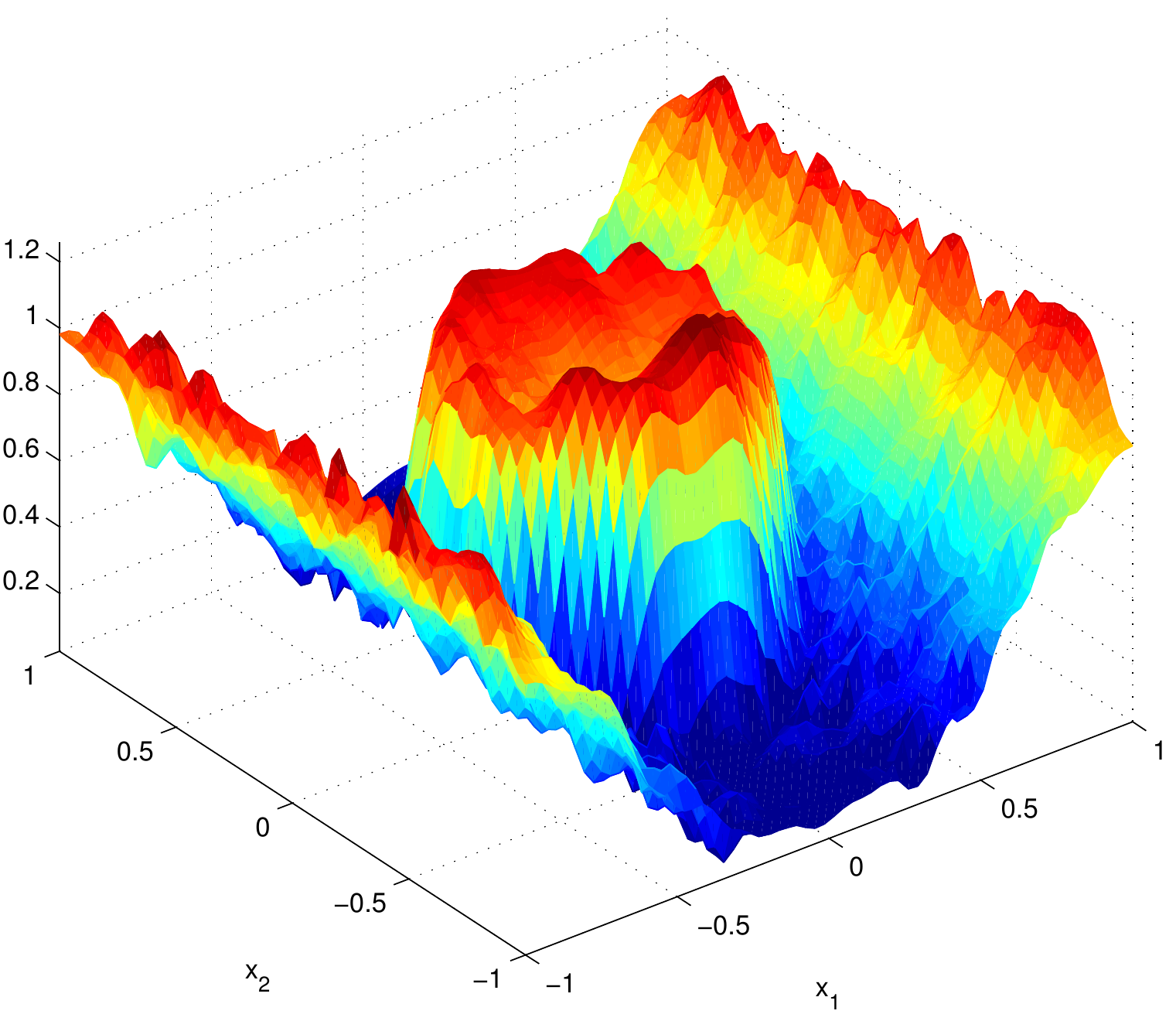}

    \subcaption{\protect\input{../figbil/op23BilData3Noise1m=1N=66}}\label{fig2:jump23}
    \label{fig:bilmultipleoperators13}

\end{subfigure}
\begin{subfigure}[t]{.5\textwidth}
\centering
\captionsetup{width=0.8\linewidth}
    \includegraphics[width=0.8\linewidth]{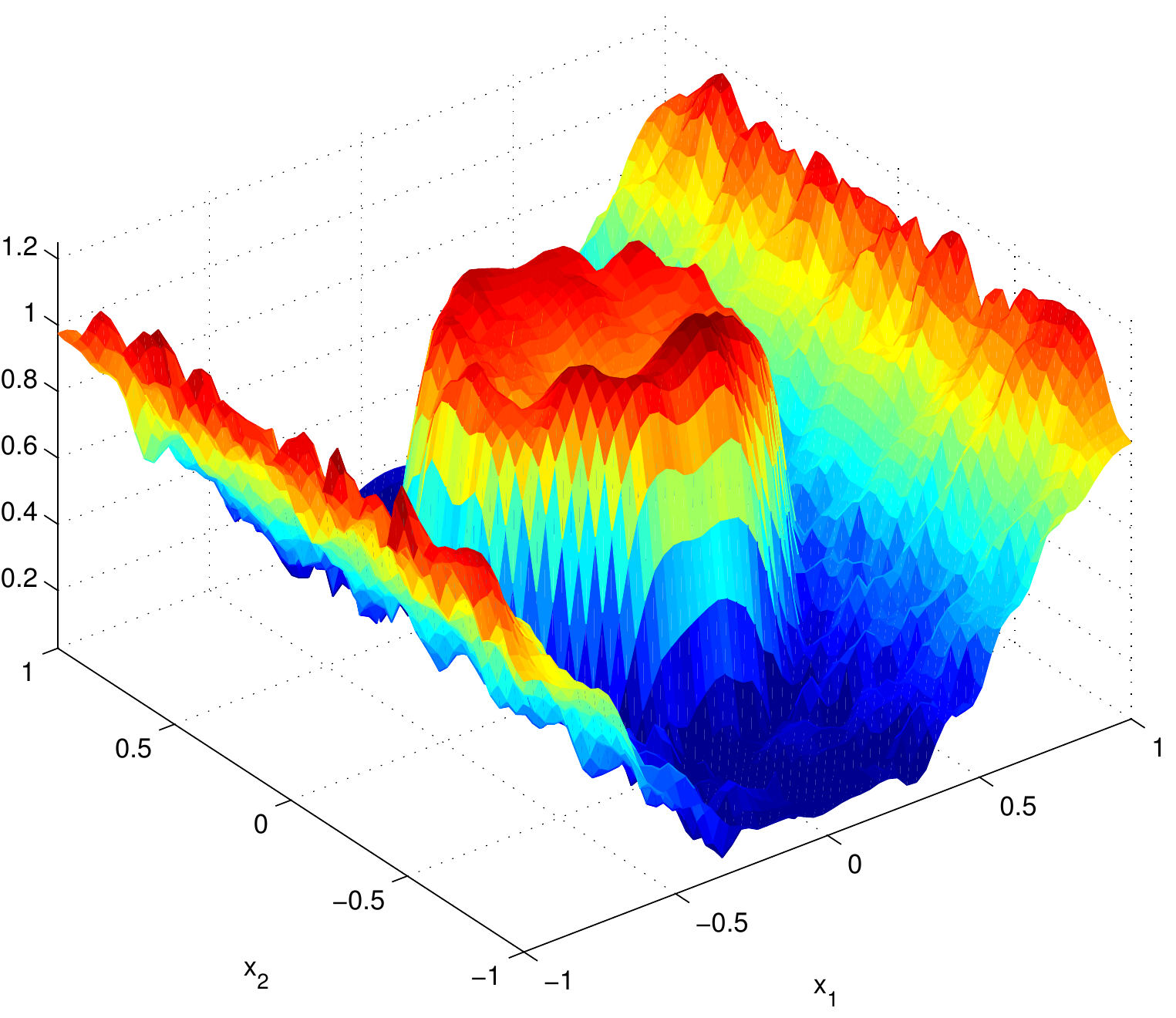}

    \subcaption{\protect\input{../figbil/op123BilData3Noise1m=1N=66}}\label{fig2:jump123}

\end{subfigure}

\caption{Optimal $u^\ast$ for the bilinear state equation, various choices of operators $K_i$ and $m=1$; $1 \%$ additive noise was used.}\label{figbil:bilresults}
\end{figure}

\begin{table}
\centering
    \begin{tabular}{| c | c | c |}
    \hline
    Used Operators & (Locally) Optimal $\alpha^\ast$ & $\|u^\ast-\uex\|_2^2$ \\ \hline
    \protect $K_1 $& $(\num{1.154829e-04 })$ & $0.73853$ \\ \hline
   \protect $K_1 $, $K_2$& $(\num{3.521353e-05 },\num{3.328048e-06 })$ & $0.23291$ \\ \hline
    \protect $K_1 $, $K_3$ & $(\num{3.939754e-05 },\num{9.219977e-07 })$ & $0.48599$ \\ \hline
    \protect $K_2$, $K_3$ & $(\num{2.485989e-07 },\num{1.523004e-07 })$ & $0.16802$ \\ \hline
    \protect $K_1 $, $K_2$, $K_3$ & $(\num{9.624362e-08 },\num{4.446735e-07 },\num{1.089828e-07 })$ & $0.16603$ \\ \hline

    \end{tabular}
     \caption{Locally optimal $\alpha^\ast$ for different sets of $K_i$ with $3\%$ noise and $m=5$.}\label{tablebil:multiple5}
	\vspace*{2\baselineskip}
\end{table}
		
\begin{table}
\centering
    \begin{tabular}{| c | c | c |}
    \hline
    Used Operators & (Locally) Optimal $\alpha^\ast$ & $\|u^\ast-\uex\|_2^2$ \\ \hline
    \protect $K_1 $& $(\num{1.217560e-04 })$ & $0.66463$ \\ \hline
   \protect $K_1 $, $K_2$& $(\num{4.447749e-05 },\num{1.376129e-06 })$ & $0.21964$ \\ \hline
    \protect $K_1 $, $K_3$ & $(\num{2.710880e-07 },\num{2.932882e-08 })$ & $0.80672$ \\ \hline
    \protect $K_2$, $K_3$ & $(\num{1.245772e-07 },\num{3.217866e-08 })$ & $0.12461$ \\ \hline
    \protect $K_1 $, $K_2$, $K_3$ & $(\num{9.816241e-08 },\num{1.233593e-07 },\num{3.620331e-08 })$ & $0.12476$ \\ \hline

    \end{tabular}
     \caption{Locally optimal $\alpha^\ast$ for different sets of $K_i$ with $3\%$ noise and $m=20$.}\label{tablebil:multiple20}

\end{table}

\FloatBarrier

\section{Outlook}

An open question which deserves to be investigated in the future is how learned regularization parameters can be used in structurally related -- but different -- problems. While in some cases learned parameters might be used directly, we suggest that in other cases they should merely be used as weights in-between multiple penalty terms, with an additional weight for the sum of all penalty terms still being determined by a classical parameter choice strategy. We also point out that the ability to compute optimal regularization parameters provides the opportunity to evaluate how well classical parameter choice strategies are performing. Another direction for further research could be to consider more general learning problems such as learning filters and  problems with non smooth lower level problems.

\appendix

\section{Proof of Lemma 5.1}

\begin{proof}
We define
\begin{equation*}
F(\alpha,y,u) \coloneqq \Jay(y,u).
\end{equation*}
In view of \cite{Zowe1979} it suffices to verify the regularity assumption consisting in the bijectivity of the mapping
\begin{equation*}
(\delta_y,\delta_u,\delta_\lambda) \to \begin{pmatrix} F_{yy} + \lambda^\ast e_{yy} & F_{yu} + \lambda^\ast e_{yu} & e_y^\ast \\
F_{yu} + \lambda^\ast e_{yu} & F_{uu} + \lambda^\ast e_{yu}  & e_u^\ast \\
e_y & e_u & 0
 \end{pmatrix}\begin{pmatrix} \delta_y \\ \delta_u \\ \delta_\lambda \end{pmatrix}
\end{equation*}
from $Y \times U \times Z\dual$ to $Y\dual \times U \dual \times Z$, where we write $F_{y y} =F_{y y}(\alpha^\ast, y^\ast, u^\ast)$, $e_{y y} =e_{yy}(\alpha^\ast,y^\ast, u^\ast)$ et cetera. This follows from observing that for every $(y',u',z) \in Y\dual \times U\dual \times Z$, the quadratic problem
\begin{align*}\label{prob:quadratic}
\begin{split}
\min_{(\delta_y,\delta_u) \in Y \times U} \ & D_{(y,u)}^2\mathcal{L}_{\alpha^\ast}(y^\ast,u^\ast,\lambda^\ast) [(\delta_y, \delta_u),(\delta_y, \delta_u)]  - y'(\delta_y) - u'(\delta_u) \quad \\
\text{subject to} \quad 
&De(y,u)(\delta_y,\delta_u) =z
\end{split} 
 \end{align*}
has a unique solution $(\delta_y^\ast, \delta_u^\ast, \delta_\lambda^\ast)$, which is characterized by 
\begin{equation*}
\begin{pmatrix} F_{yy} + \lambda e_{yy} & F_{yu} + \lambda^\ast e_{yu} & e_y^\ast \\
F_{yu} + \lambda^\ast e_{yu} & F_{uu} + \lambda^\ast e_{yu}  & e_u^\ast \\
e_y & e_u & 0
 \end{pmatrix} \begin{pmatrix} \delta_y^\ast \\ \delta_u^\ast \\ \delta_\lambda^\ast \end{pmatrix}= \begin{pmatrix} y' \\ u' \\ z \end{pmatrix}.
\end{equation*}
\end{proof}

\section{Proof of Theorem 5.1}

\begin{proof}
We define
\begin{equation*}
F(\alpha,y,u) \coloneqq \Jay(y,u).
\end{equation*}
Recall that since $(y^\ast,u^\ast)$ is a solution to $(\Lpayast)$, and $e_y(y^\ast,u^\ast)$ is bijective, there exists a unique $\lambda^\ast \in Z\dual$ such that
\begin{align*}
\begin{split}       
 F_y(\alpha^\ast,y^\ast,u^\ast) + \lambda^\ast e_y(y^\ast,u^\ast) &=0, \\
 F_u(\alpha^\ast,y^\ast,u^\ast) + \lambda^\ast e_u(y^\ast,u^\ast) &=0, \\
e(y,u)&=0.          \\
       \end{split}
\end{align*}
As in the proof of \Cref{opt:biopt}, one can show bijectivity of the mapping
\begin{equation*}
(\delta_y,\delta_u,\delta_\lambda) \to \begin{pmatrix} F_{yy} + \lambda^\ast e_{yy} & F_{yu} + \lambda^\ast e_{yu} & e_y^\ast \\ 
F_{yu} + \lambda^\ast e_{yu} & F_{uu} + \lambda^\ast e_{yu}  & e_u^\ast \\
e_y & e_u & 0 
 \end{pmatrix} \begin{pmatrix} \delta_y \\ \delta_u \\ \delta_\lambda \end{pmatrix}.
\end{equation*}
Thus, by the implicit function theorem, there exists neighbourhoods $I$ of $\alpha^\ast$ and $V$ of $(y^\ast,u^\ast,\lambda^\ast)$ and a continuously F-differentiable function $\Phi \colon I \to V$ such that for all $\alpha \in I$ and $(y,u,\lambda) \in V $ it holds that
\begin{align*}
\begin{split}       
 F_y(\alpha,y,u) + \lambda e_y(y,u) &=0, \\
 F_u(\alpha,y,u) + \lambda e_u(y,u) &=0, \\
e(y,u)&=0,          \\
       \end{split}
\end{align*}
if and only if
\begin{equation*}
\Phi(\alpha)=(y,u,\lambda).
\end{equation*}
A standard argument can be used to show that $I$ can be chosen such that the second order sufficient optimality conditions of \eqref{prob:llpproblem} still hold in $\Phi(\alpha)=(y_\alpha,u_\alpha,\lambda_\alpha) $ for every $\alpha \in I$. Consequently,
\begin{equation*}
(y_\alpha,u_\alpha,\lambda_\alpha)
\end{equation*} 
is a local solution to lower level problem for every $\alpha \in I$. We now claim that there exists a neighbourhood $J$ of $\alpha^\ast$ contained in  $I$ such that $\phi(\alpha)$ is a global solution to the lower level problem for every $\alpha \in J$. We prove this by contradiction. If our claim was false, there would be a sequence $(\alpha^n)$ in  $I$ such that
\begin{equation*}
\alpha^n \to \alpha^\ast 
\end{equation*}
with an associated sequence $(y^n,u^n,\lambda^n)$ of solutions to $(\Lpayno)$, which does not intersect $V$.
Using the stability assumption, the uniqueness assumption, and that $e_y(y^\ast,u^\ast)$ is bijective, it is straightforward to see that $(y^n, u^n, \lambda^n)$ must converge to $(y^\ast,u^\ast, \lambda^\ast)$. However since the sequence was chosen such that $(y^n,u^n,\lambda^n) \notin V$ for all $n \in \N$, this leads to a contradiction. This shows that $(\alpha^\ast,y^\ast,u^\ast, \lambda^\ast)$ is a solution to \eqref{prob:genbioptrel} restricted to $J \times V$, \ie a local solution to \eqref{prob:genbioptrel}. This completes the proof.

\end{proof}

\begin{rem*} We point out that the proofs of \Cref{opt:biopt} and \Cref{thm:localsolution} do not rely on the specific structure of the learning problem. Thus, these results can be applied to more general bilevel optimization problems of the form
\begin{equation*}
\min_{(\alpha, y_\alpha,u_\alpha) \in C \times Y \times U } G(\alpha,y_\alpha,u_\alpha) \quad \sto \quad (y_\alpha,u_\alpha) \in \argmin_{(y,u) \in Y \times U }\{ F(\alpha,y,u) \mid e(y,u)=0 \},
\end{equation*}
for given $G,F \colon X \times Y \times U \to \R$, $e \colon Y \times U \to Z$, and a closed and convex set $C \subseteq X$, where $Y,U$ are Hilbert spaces, and $X,Z$ are Banach spaces.

\end{rem*}

\section{Used functions}

The function $\phi \colon \R \to \R$ is defined as follows.
\begin{equation*}
\phi(t) \coloneqq \begin{cases}
t \quad &\text{for} \quad a+\epsilon \leq t \leq b-\epsilon\\
a \quad &\text{for} \quad t \leq a\\
b \quad &\text{for} \quad b \leq t\\
f(t-a) \quad &\text{for}  \quad a \leq t \leq a+\epsilon \\
-f(-t+b)+a+b \quad &\text{for} \quad b-\epsilon \leq t \leq b,
\end{cases}
\end{equation*}
where $\varepsilon >0$ is a small parameter and 
\begin{equation*}
f(x)\coloneqq \frac{-10}{\varepsilon^6} x^7 + \frac{36}{\varepsilon^5}x^6 -\frac{45}{\epsilon^4}x^5 + \frac{20}{\varepsilon^3} x^4 + a \quad \text{for } x \in \R.
\end{equation*}
In particular, one can verify that $\phi \in C^3(\R,\R)$.

\bibliography{ms}

\begin{thebibliography}{10}

\bibitem{adams1975sobolev}
R.~A. Adams.
\newblock {\em Sobolev Spaces}.
\newblock Pure and applied mathematics. Academic Press, 1975.

\bibitem{tikhonov1977solutions}
V.~Y. Arsenin and A.~N. Tikhonov.
\newblock {\em Solutions of ill-posed problems}, volume~14.
\newblock Winston Washington, DC, 1977.

\bibitem{brezis2010functional}
H.~Brezis.
\newblock {\em Functional Analysis, Sobolev Spaces and Partial Differential
  Equations}.
\newblock Universitext. Springer New York, 2010.

\bibitem{ochs2015bilevel}
T.~Brox, P.~Ochs, T.~Pock, and R.~Ranftl.
\newblock Bilevel optimization with nonsmooth lower level problems.
\newblock In {\em International Conference on Scale Space and Variational
  Methods in Computer Vision}, pages 654--665. Springer, 2015.

\bibitem{chung2017}
J.~Chung and M.~I. Español.
\newblock Learning regularization parameters for general-form {Tikhonov}.
\newblock {\em Inverse Problems}, 33(7):074004, 2017.

\bibitem{de2013image}
J.~C. De~los Reyes and C.-B. Sch{\"o}nlieb.
\newblock Image denoising: Learning the noise model via nonsmooth
  {PDE}-constrained optimization.
\newblock {\em Inverse Problems \& Imaging}, 7(4), 2013.

\bibitem{DELOSREYES2016464}
J.~C. De~los Reyes, C.-B. Sch{\"o}nlieb, and T.~Valkonen.
\newblock The structure of optimal parameters for image restoration problems.
\newblock {\em Journal of Mathematical Analysis and Applications}, 434(1):464
  -- 500, 2016.

\bibitem{dempe2002foundations}
S.~Dempe.
\newblock {\em Foundations of bilevel programming}.
\newblock Springer Science \& Business Media, 2002.

\bibitem{engl1996regularization}
H.~W. Engl, M.~Hanke, and A.~Neubauer.
\newblock {\em Regularization of Inverse Problems}.
\newblock Mathematics and Its Applications. Springer Netherlands, 1996.

\bibitem{HollerThesis2017}
G.~Holler.
\newblock {A bilevel approach for parameter learning in inverse problems}.
\newblock Master's thesis, Karl-Franzens Universität Graz, Austria, 2017.

\bibitem{kunischsequential}
K.~Ito and K.~Kunisch.
\newblock A sequential method for mathematical programming.

\bibitem{jahn2007introduction}
J.~Jahn.
\newblock {\em Introduction to the theory of nonlinear optimization}.
\newblock Springer Science \& Business Media, 2007.

\bibitem{kaltenbacher2008iterative}
B.~Kaltenbacher, A.~Neubauer, and O.~Scherzer.
\newblock {\em Iterative regularization methods for nonlinear ill-posed
  problems}, volume~6.
\newblock Walter de Gruyter, 2008.

\bibitem{Kunisch_Bilevel2013}
K.~Kunisch and T.~Pock.
\newblock A bilevel optimization approach for parameter learning in variational
  models.
\newblock {\em SIAM Journal on Imaging Sciences}, 6(2):938--983, 2013.

\bibitem{Zowe1979}
S.~Kurcyusz and J.~Zowe.
\newblock Regularity and stability for the mathematical programming problem in
  {Banach} spaces.
\newblock {\em Applied Mathematics and Optimization}, 5(1):49--62, 1979.

\bibitem{louis2013inverse}
A.~K. Louis.
\newblock {\em Inverse und schlecht gestellte {Probleme}}.
\newblock Springer-Verlag, 2013.

\bibitem{Maurer1979}
H.~Maurer and J.~Zowe.
\newblock First and second-order necessary and sufficient optimality conditions
  for infinite-dimensional programming problems.
\newblock {\em Mathematical Programming}, 16(1):98--110, Dec 1979.

\bibitem{ulbrich2012nichtlineare}
M.~Ulbrich and S.~Ulbrich.
\newblock {\em Nichtlineare Optimierung}.
\newblock Springer-Verlag, 2012.

\end{thebibliography}
\bibliographystyle{plain}

\end{document}